\documentclass[10pt,a4paper]{amsart}
\usepackage[cp1250]{inputenc}
\usepackage[english]{babel}
\usepackage{color}
\usepackage[colorlinks,linkcolor=blue,citecolor=red]{hyperref}
\usepackage{amssymb}
\usepackage{amsmath}
\usepackage{amsfonts}
\usepackage[IL2]{fontenc}
\usepackage{url}
\usepackage{graphicx}

\usepackage{pgf,tikz}
\usepackage{mathrsfs}
\usetikzlibrary{arrows}
\usetikzlibrary{snakes}
\usetikzlibrary{arrows}
\usetikzlibrary{calc}
\usetikzlibrary{patterns}
\usepackage{tkz-fct}

\newtheorem{defi}{\bf D\scriptsize EFINITION \normalsize}

\newtheorem{lm}{\bf L\scriptsize EMMA \normalsize}
\newtheorem{dk}{\bf C\scriptsize OROLLARY \normalsize}
\newtheorem{rem}{\bf R\scriptsize EMARK \normalsize}
\newtheorem{exa}{\bf E\scriptsize XAMPLE \normalsize}
\newtheorem{pro}{\bf P\scriptsize ROBLEM \normalsize}
\newtheorem{prop}{\bf P\scriptsize ROPOSITION \normalsize}
\newtheorem{no}{\bf N\scriptsize OTE \normalsize}

\newenvironment{remark}{\begin{rem}\rm}{\end{rem}}
\def\kopr{\hfill\raisebox{3pt}{\framebox{$\star$}}}
\newenvironment{example}{\begin{exa}\rm}{$\kopr$\end{exa}}
\newenvironment{definition}{\begin{defi}\rm}{\end{defi}}

\newenvironment{lemma}{\begin{lm}\it}{\end{lm}}
\newenvironment{corollary}{\begin{dk}\it}{\end{dk}}
\newenvironment{proposition}{\begin{prop}\it}{\end{prop}}

\newcommand{\inte}[2]{\int \limits_{#1}^{#2}}

\newcommand{\nadsebou}[2]{\begin{array}{c} #1 \\ #2 \end{array}}

\newcommand{\hzav}[1]{\left[ #1 \right]}
\newcommand{\zav}[1]{\left( #1 \right)}

\newcommand{\hypzav}[1]{\left. #1 \right|}
\newcommand{\imag}{{\rm i}}

\newcommand{\abs}[1]{\left| #1 \right|}
\renewcommand{\arctan}[1]{{\rm arctan}#1}
\renewcommand\Re{\operatorname{Re}}
\renewcommand\Im{\operatorname{Im}}
\allowdisplaybreaks
\numberwithin{equation}{section}
\begin{document}
\title{Hypergeometric form of Fundamental theorem of calculus}
 \author{Petr Blaschke}
 \thanks{Supported by GA \v CR grant no. 201/12/G028}
\address{ Mathematical Institute, Silesian University in Opava, Na Rybnicku 1, 746 01 Opava, Czech Republic}
\email{Petr.Blaschke@math.slu.cz}
\begin{abstract} 
We introduce a natural method of computing antiderivatives of a large class of functions (holomorphic near the origin) which stems from the observation that the series expansion of an antiderivative differs from the series expansion of the corresponding integrand by just two Pochhammer symbols. All antiderivatives are thus, in a sense, ``hypergeometric''. And hypergeometric functions are therefore the most natural functions to integrate. This paper would like to make two points: First, the method presented is \textit{easy}. So much so that it can be taught in undergraduate university level. And second: It may be used to prove some of the more challenging examples computed only by heuristic processes like Method of brackets.
\end{abstract}
\maketitle
\section{Introduction}\label{Intro}
The oldest method of evaluating definite integrals is the ``Fundamental theorem of calculus'', i.e. the fact that
$$
\inte{a}{b}f'(x){\rm d}x=f(a)-f(b).
$$

If the antiderivative is representable in terms of elementary functions, this is, indeed, way to go. But in the opposite case (which is generic), situation becomes much more complicated.

G. Cherry in his works \cite{Cherry1},\cite{Cherry2},\cite{Cherry3}, made an interesting point that even though the antiderivative of
$$
\frac{x^3}{\ln(x^2-1)},
$$ 
is representable in terms of elementary functions and the logarithmic integral function li$(x):=\int \frac{1}{\ln x}{\rm d}x$, specifically
$$
\int \frac{x^3}{\ln(x^2-1)}{\rm d}x=\frac12 {\rm li}\zav{(x^2-1)^2}+\frac12 {\rm li}(x^2-1)+c,
$$
the same cannot be accomplished for the nearly the same function
$$
\frac{x^2}{\ln(x^2-1)},
$$  
and new special function is required.

This example shows that it is very difficult beforehand to decide on the class of functions in which the antiderivative should be looked for. Existence of a single, almost universally useful class seems quite impossible.

We can even say that, for this reason, the  approach of antiderivatives somewhat falls out of grace of the mathematical community and -- going as back as Cauchy -- people start treating definite integrals as an independent objects, looking for ways how to evaluate them without computing antiderivatives.

There are, for example, at least 4 methods (known to the author) how to evaluate the famous definite integral
$$
\inte{-\infty}{\infty}\frac{\sin x}{x}{\rm d}x,
$$    
none of which deals with the antiderivative, since it is not elementary. Those methods are, in turn: Laplace transform, Fourier transform, calculus of residues and -- the newest addition to the subject -- the ``method of brackets'' developed in \cite{mob}.

While these approaches works quite handsomely and even elegantly, they do not exactly satisfy a freshmen idea of ``simple''. Great care about various convergence issues has to be taken care of (especially with complex contour integration in the calculus of residues) and the method of brackets is not even rigorous (yet). 

In addition, these methods varies greatly performance-vise from integral to integral. Instead of learning a single superior method, one has to master them all to be effective. 

For the integral above, to illustrate this point, is perhaps the Fourier transform the most convenient tool to pick, even more so for the related integral
$$
\inte{-\infty}{\infty}\zav{\frac{\sin x}{x}}^2{\rm d}x,
$$
but the Fourier transform is absolutely teethless, dealing with gaussian-like integrals:
$$
\inte{-\infty}{\infty} e^{-x^2}{\rm d}x, \inte{-\infty}{\infty} e^{-x^3}{\rm d}x, \dots
$$ 
for which the calculus of residue is more suited (leaving the method of brackets aside for the moment). 

Of course, there is also Gauss's own approach for the first integral in the above list -- converting it into a double integral by taking its square. This is perhaps the cleverest trick there is regarding integration. A spectacular example of ``outside the box'' (or may be ``outside the dimension'') thinking. But for its ingenuity it is remarkable limited in its use since it does not appear -- to the author knowledge at least -- anywhere else. Thus, another method to learn.

\bigskip
We are going to argue that those integrals (and many more) can be computed using simple, single procedure that would please any freshmen because it does not differ from the Fundamental theorem of calculus.

Part of this procedure is the realization that there \textit{is} a large family of functions which has amazing properties concerning integration. It is as if it was made for it (but, actually, it was not). These properties includes:
\begin{itemize}
\item{The family is closed under the operation making antiderivatives.}
\item{Contains many elementary functions, especially the ``troublesome'' ones like}
$$
\frac{\sin x}{x},\zav{\frac{\sin x}{x}}^2,\frac{\arctan x}{x},e^{x^2},e^{x^3},\dots
$$
\item{There are \textit{multiple} ways how to represent each function, which allows one to make connections between integrals that links no conceivable change of variable.}
\item{And, finally, computing the antiderivate is very simple. It requires no calculus, no algebra, in fact, no mental effort at all. Just adding two parameters.}
\end{itemize}

This wonderful family is, of course, generalized hypergeometric functions $\!\! \ _p F_q$ and their antiderivatives can be computed as follows:
$$
\int x^\alpha \!\! \ _p F_q\zav{\nadsebou{a_1 \dots a_p}{c_1 \dots c_q};\gamma x^\beta}{\rm d}x=\frac{x^{\alpha+1}}{\alpha+1} \!\! \ _{p+1} F_{q+1}\zav{\nadsebou{a_1 \dots a_p\quad \frac{\alpha+1}{\beta}}{c_1 \dots c_q\quad 1+\frac{\alpha+1}{\beta}};\gamma x^\beta}+c.
$$
We will make this precise in Section \ref{S3}.

Even more function can be represented with various multi-variable hypergeometric functions.

But instead of just moving to higher dimensions, we are going to employ the concept of ``hypergeometrization'' which we briefly introduce in Section \ref{S2}. This is very convenient tool how to speak about hypergeometric functions. It enables the user to make statements about them which includes various numbers of parameters and various dimension simultaneously, instead of making a statement for each case separately (as we will see). 

Most importantly, it will enable us to (symbolically) integrate \textit{all} functions (holomorphic near the origin).

Formally, we prove the following:
\begin{proposition}\label{HFFTC} (Hypergeometric form of fundamental theorem of calculus) Let $f$ be a function holomorphic near $0$, $\alpha\not=-1,\beta\not=0$, $-\frac{\alpha+1}{\beta}\not\in\mathbb{N}$. Then
\begin{equation}\label{HFFTCformula}
\int x^\alpha f\zav{ x^\beta}{\rm d}x=\frac{x^{\alpha+1}}{\alpha+1} f\zav{\hypzav{\nadsebou{\frac{\alpha+1}{\beta}}{1+\frac{\alpha+1}{\beta}}} x^{\beta}}+c,
\end{equation}
\end{proposition}
where
$$
f\zav{\hypzav{\nadsebou{a}{c}};x}:=\sum_{k=0}^{\infty}\frac{f^{(k)}(0)}{k!}\frac{(a)_{k}}{(c)_k}x^k.
$$
This is based on the observation that what is true for hypergeometric function, is true for all -- that making the antiderivative just adds two Pochhammer symbols to the Taylor series.

We will also prove the exceptional case:
\begin{proposition}\label{HFFTCln} Let $f$ be holomorphic near origin, $\alpha\not=0$. Then
$$
\int\frac{1}{x}f(x^\alpha){\rm d}x=f(0)\ln x+\frac{f(x^\alpha)-f(0)}{\alpha}-\frac{x^\alpha}{\alpha}\hzav{\epsilon}f'\zav{\hypzav{\nadsebou{1+\epsilon}{2+\epsilon}}x^\alpha}+c,
$$
\end{proposition}
where here and throughout the paper $\hzav{\epsilon^k}$ denotes the $k$-th Taylor coefficient of the function to the right, i.e.
$$
[\epsilon^k]f(\epsilon)=\frac{f^{(k)}(0)}{k!}.
$$

Consequences of these theorems would be illustrated on number of examples. Among other we are going to show that:
\begin{equation}\label{maplenogo}
\inte{0}{\infty}\sqrt[8]{\frac{x^2+8x+8-4(2+x)\sqrt{1+x}}{x^{11}}}{\rm d}x=\frac{4 \Gamma^2\zav{\frac14}}{3\sqrt{2-\sqrt{2}}\sqrt{\pi}},
\end{equation}
a result that cannot be obtained using the mathematical software MAPLE 2016. 

Also to demonstrate possible impact in Number theory, we are going to re-derive known representations of Catalan's constant $G$:
$$
G=\Re \zav{\!\! \ _3 F_2\zav{\nadsebou{1\quad 1\quad 1}{2\quad 2};\imag}}=\Im\zav{\hzav{\epsilon^2}\!\! \ _2 F_1\zav{\nadsebou{\epsilon\quad \epsilon}{1};\imag}}=\frac{1}{8}\zav{\psi'\zav{\frac14}-\pi^2}.
$$

Of course, not \textit{all} functions are representable by $\!\! \ _p F_q$. But even in this case all is not lost. We can either perform some change of variable which brings the integrand to one that \textit{is} representable by a hypergeometric function or  -- which is one of the main points of this article -- we can often represent it as a hypergeometric function \textit{with differentiated parameters}.
 
For instance, together with many other examples we will show in Section \ref{S4} that
$$
\zav{\frac{{\rm arcsin} x}{x}}^3=-\frac32  \hzav{\epsilon^2} x^{-2}\!\! \ _2 F_1\zav{\nadsebou{\frac12-\epsilon\quad \frac12+\epsilon}{\frac32};x^2}. 
$$

From this representation it is easy to compute:
$$
\inte{0}{1}\zav{\frac{{\rm arcsin}\, x}{x}}^3{\rm d}x=\frac32\pi\ln 2-\frac{\pi^3}{16},
$$
a result that is not impossible to derive by other means but our method is completely straightforward and direct once this representation is known.

Also, there are  some general rules regarding hypergeometric function with differentiated parameters we are going to present.

\begin{proposition}\label{genrule1prop}
 For $x\not \in [1,\infty)$:
\begin{align*}
\hzav{\epsilon}\!\! \ _2 F_1\zav{\nadsebou{a+\epsilon\quad b+\epsilon}{a+b};x}&=\ln\frac{1}{1-x}\!\! \ _2 F_1\zav{\nadsebou{a\quad b}{a+b};x}.
\end{align*}
\end{proposition}
This fact will enable us to compute the integrals
$$
\inte{0}{\infty}\frac{\arctan\, x}{x^{2\alpha+1}}\ln\frac{1}{1+x^2}{\rm d}x=\frac{\pi}{4\alpha\cos(\pi\alpha)}\zav{\psi\zav{\frac12+\alpha}+\psi(\alpha)-\psi(1)-\psi\zav{\frac12}},\qquad  0<\alpha<\frac12,
$$
$$
\inte{0}{1}x\ln\frac{1}{1-x^2}K(x){\rm d}x=4(1-\ln 2),
$$
and
$$
\inte{0}{1}x\ln\frac{1}{1+x^2}K(\imag x){\rm d}x=\frac{1}{4\sqrt{2\pi}}\zav{(2-\ln 2)\Gamma^2\zav{\frac14}+4(\ln 2-4)\Gamma^2\zav{\frac34}},
$$
where
$$
 K(x):=\frac{\pi}{2}\!\! \ _2 F_1\zav{\nadsebou{\frac12\quad \frac12}{1};x^2},
$$
is the complete elliptic integral of the first kind.

This is the part of the paper that the author considers most novel.

\subsection{Relationship to Method of brackets} 
Method of brackets deals only with integrals of the form $\inte{0}{\infty}$, it does not compute antiderivatives and (unlike our approach) is not rigorous (yet).

But there is a great deal of similarity between it and our approach in the sense that both methods are chiefly preoccupied with the series expansion of the integrand.

We do not claim that our approach puts the method of brackets on rigorous ground in any way but we believe that it should be effective for similar types of integral.

Method of brackets works basically by ignoring convergence issues in swapping the series expansion and the integration, and only tries to resolve the divergences by some heuristic process, solving a system of linear equations (which works surprisingly well). We instead compute the antiderivative for which such a swapping is no issue and then convert the problem of evaluation a definite integral to a problem of evaluation a (possibly multi-dimensional) hypergeometric function in specific arguments.

The point is that hypergeometric functions are studied several centuries by known and the acquired knowledge is vast, literature numerous (for example \cite{Erdelyi},\cite{Luke},\cite{Lauricella},\cite{Srivastava},\cite{Olver}). All of this can be used to our benefit.

Particularly useful in this sense are websites: \textit{Digital Library of Mathematical Functions} \cite{dlmf}, which provides an easy access to numerous identities valid for hypergeometric functions. We will refer to them multiple times.

To really test our method we are going to concentrate on the integral:
$$
I_{\frac12}:=\inte{0}{\infty}\frac{1}{\zav{1+x^2}^{\frac32}\sqrt{\varphi(x)+\sqrt{\varphi(x)}}}{\rm d}x,\qquad \varphi(x):=1+\frac43\frac{x^2}{(1+x^2)^2},
$$
which has been incorrectly evaluated (as mentioned in \cite{mob}) in the table of definite integrals \cite{GR} to be equal to
$$
\frac{\pi}{2\sqrt{6}}.
$$
We will show that, in fact, this is a correct value for a very similar integral:
\begin{proposition}\label{P1} For $\varphi(x)$ as above it holds: 
$$
I_{true}:=\inte{0}{\infty}\frac{1}{\zav{1+x^2}^{\frac32}\sqrt{\varphi(x)+\sqrt{\varphi(x)^3}}}{\rm d}x=\frac{\pi}{2\sqrt{6}}.
$$
\end{proposition}
which is, perhaps, what the original input in \cite{GR} should have been. This misprint is easy to spot with our method.

In fact, defining: 
$$
\varphi_\alpha(x):=1+4\alpha^2\frac{x^2}{(1+x^2)^2},
$$
we can obtain general result:
\begin{proposition}\label{P2} Let $\varphi_\alpha(x)$ as above. Then
$$
I_{\alpha,true}:=\inte{0}{\infty}\frac{1}{\zav{1+x^2}^{\frac32}\sqrt{\varphi_\alpha(x)+\sqrt{\varphi_\alpha(x)^3}}}{\rm d}x=\frac{\arctan\,\alpha}{\sqrt{2}\alpha}.
$$
\end{proposition}
Integral $I_{true}$ is a special case of this for $\alpha=\frac{1}{\sqrt{3}}$.

Unfortunately, we too were unable to evaluate the integral $I_{\frac12}$ itself. We are only able to represent it as a value of a two-variable hypergeometric function:
$$
I_\frac12=\frac{1}{\sqrt{2}} \tilde F_1\zav{\nadsebou{1\quad \frac12}{\frac34\quad \frac54};\nadsebou{\frac{1}{4}\quad \frac{3}{4}}{\frac32}\nadsebou{\frac{1}{4}}{-};-\frac13,-\frac13},
$$
where
$$
\tilde F_1\zav{\nadsebou{\alpha_1\quad \alpha_2}{\gamma_1\quad \gamma_2};\nadsebou{a_1\quad a_2}{c}\nadsebou{b}{-};tx,ty}:=\sum_{j,k=0}^{\infty}\frac{(\alpha_1)_{j+k}(\alpha_2)_{j+k}}{(\gamma_1)_{j+k}(\gamma_2)_{j+k}}\frac{(a_1)_j(a_2)_j}{(c)_j j!}\frac{(b)_k}{k!}t^{j+k}x^jy^k,
$$
can be regarded as a generalization of Appell's $F_1$ function. Such a representation is of limited use, however.

Natural thing to do is to look for similar integrals to see if they are any more approachable than the original problem.  
We are thus going to concentrate (as an illustration of what our method can and cannot do) on the family of integrals:
$$
I_\alpha:=\inte{0}{\infty}\frac{1}{\zav{1+x^2}^{\frac32}}\zav{\varphi(x)+\sqrt{\varphi(x)}}^{-\alpha}{\rm d}x,\qquad \varphi(x):=1+\frac43\frac{x^2}{(1+x^2)^2},
$$
that can be also represented by the same $\tilde F_1$ function as $I_\frac12$, specifically:
$$
I_\alpha=2^{-\alpha} \tilde F_1\zav{\nadsebou{1\quad \frac12}{\frac34\quad \frac54};\nadsebou{\frac{\alpha}{2}\quad \frac{\alpha+1}{2}}{\alpha+1}\nadsebou{\frac{\alpha}{2}}{-};-\frac13,-\frac13}.
$$

We are going to show that this is a tough family of integrals, indeed. For no value of $\alpha$ (save for $\alpha=0$ which is trivial) we were able to evaluate $I_\alpha$ in terms of elementary function and $\Gamma$ function and
only for the values $I_0,I_1,I_{-1},I_{-2},[\alpha] I_{\alpha}$ we were able to evaluate in terms of single-variable hypergeometric function.
\begin{proposition}\label{Ialpha} Let $I_\alpha$ be as above. Then:
\begin{align}
I_0&=1,\\
I_1&=\frac12\!\! \ _4 F_3\zav{\nadsebou{1\quad \frac12\quad\frac32\quad 1}{2\quad \frac34\quad \frac54};-\frac13}=-\frac38 \hzav{\epsilon} \!\! \ _3 F_2\zav{\nadsebou{\epsilon \quad -\frac12\quad\frac12}{ -\frac14\quad \frac14};-\frac13},\\
I_{-1}&=\!\!\ _3 F_2\zav{\nadsebou{1\quad \frac12\quad -\frac12}{\frac34\quad \frac54};-\frac13}+\frac{53}{45},\\
I_{-n}&=\sum_{k=0}^{n}\zav{\nadsebou{n}{k}}\!\! \ _3 F_2\zav{\nadsebou{-\frac{n+k}{2}\quad 1\quad \frac12}{\frac34\quad \frac54};-\frac13},\\
I_2&=\frac{3\sqrt{2}}{8}\arctan\sqrt{2}+\frac{\sqrt{6}}{16}\ln\zav{5+2\sqrt{6}}-\frac34 \!\! \ _4
 F_3\zav{\nadsebou{1\quad 1\quad \frac12\quad \frac52}{\frac34\quad \frac54\quad 3};-\frac13},\\
\hzav{\epsilon}I_\epsilon&=\ln\frac12+2-\frac{\sqrt{2}}{2}\arctan\sqrt{2}-\frac{\sqrt{6}}{4}\ln\zav{5+2\sqrt{6}}-\frac{2}{45}\!\! \ _4 F_3\zav{\nadsebou{1\quad 1\quad \frac32\quad \frac32}{2\quad\frac74\quad \frac94};-\frac13}.
\end{align}
\end{proposition}
This will be shown in Section \ref{S5}.

The problem of integration is perhaps \textit{the} oldest problem of modern mathematical analysis, yet it still attracts attention of many researchers (\cite{mobref1},\cite{mobref2},\cite{mobref3} and, of course, Method of brackets \cite{mob}) and it is a surprisingly active area.

\section{hypergeometrization}\label{S2}
\begin{definition}
Let $f$ be a function holomorphic near $0$, $1-c\not \in \mathbb{N}$ then its \textit{hypergeometrization} is defined by the infinite series:
$$
f\zav{\hypzav{\nadsebou{a}{c}}x}:=\sum_{k=0}^{\infty}\frac{f^{(k)}(t)}{k!}\frac{(a)_k}{(c)_k}x^k,
$$
where $(a)_k=a(a+1)(a+2)\cdots (a+k-1)$ is the Pochhammer symbol.
\end{definition} 
\begin{remark}
\ 
\begin{itemize}
\item{Obviously}
$$
f\zav{\nadsebou{a}{a};x}=f\zav{x}.
$$
\item{For $a=-n$, $n+1\in \mathbb{N}$, the series terminates and the function
$$
f\zav{\nadsebou{-n}{c};x}=\sum_{k=0}^{n}\frac{f^{(k)}(0)}{k!}\frac{(-n)_k}{(c)_k}x^k,
$$
is a polynomial in $x$.
}
\item{ Taylor series $f\zav{x}$ around $x$ and $f\zav{\nadsebou{a}{c};x}$ converges in the same disk.}
\item{We will write iterative application of hypergeometrization in a more compact way:
$$
f\zav{\hypzav{\nadsebou{a\quad b}{c\quad d}}x}:=f\zav{\hypzav{\nadsebou{a}{c}}\hypzav{\nadsebou{b}{d}}x},
$$}
and so on. 
\item{Hypergeometrization can undo itself, i.e.}
$$
f\zav{\hypzav{\nadsebou{a}{c}\nadsebou{c}{a}}x}=f(x).
$$
\item{The resulting function does not depend on the order of operation:}
$$
f\zav{\hypzav{\nadsebou{a}{c}\nadsebou{b}{d}}x}=f\zav{\hypzav{\nadsebou{b}{c}\nadsebou{a}{d}}x}=f\zav{\hypzav{\nadsebou{a}{d}\nadsebou{b}{c}}x}=f\zav{\hypzav{\nadsebou{b}{d}\nadsebou{a}{c}}x}.
$$
\item{From the property of Pochhammer symbol: 
$$
(a)_{2k}=\zav{\frac{a}{2}}_k\zav{\frac{a+1}{2}}_k 4^{k},
$$}
we can see that for $f(x)=g(x^2),$ it holds
\begin{equation}\label{Pochsqr}
f\zav{\hypzav{\nadsebou{a}{c}}x}=g\zav{\hypzav{\nadsebou{\frac{a}{2}\quad \frac{a+1}{2}}{\frac{c}{2}\quad \frac{c+1}{2}}}x^2}.
\end{equation}
\item{Generally, for $n\in \mathbb{N}$ it holds
$$
(a)_{nk}=\zav{\frac{a}{n}}_k\zav{\frac{a+1}{n}}_k\cdots \zav{\frac{a+n-1}{n}}_k n^{nk},
$$}
and thus for $f(x)=g(x^n)$ we have
$$
f\zav{\hypzav{\nadsebou{a}{c}}x}=g_0\zav{\hypzav{\nadsebou{\frac{a}{n}\quad \frac{a+1}{n}\dots \frac{a+n-1}{n}}{\frac{c}{n}\quad \frac{c+1}{n}\dots \frac{c+n-1}{n}}}x^n}.
$$
%
\end{itemize}

\end{remark}

\bigskip
We can represent many special functions as a hypergeometrization of elementary functions:
\begin{align}
\intertext{Confluent hypergeometric function from $f(x):=e^x$:}
f\zav{\hypzav{\nadsebou{a}{c}}x}&=:\!\! \ _1 F_1\zav{\nadsebou{a}{c};x}.\\
\intertext{Gauss's hypergeometric function from $f(x):=(1-x)^{-b}$:}
f\zav{\hypzav{\nadsebou{a}{c}}x}&=:\!\! \ _2 F_1\zav{\nadsebou{a\quad b}{c};x}. \label{gauss}\\
\intertext{Bessel function of the first kind from $f(x):=\cos(2\sqrt{x})=\sum\limits_{k=0}^{\infty}\frac{(-4x)^k}{(2k)!}$:}
f\zav{\hypzav{\nadsebou{\frac12}{c}}x}&=: \!\! \ _0 F_1\zav{\nadsebou{-}{c};-x}=\Gamma(c)x^{\frac{1-c}{2}} J_{c-1}(2\sqrt{x}).\\
\intertext{Generelized hyp. function from $f(x):=\frac13\zav{e^{3\sqrt[3]{x}}+2e^{-\frac32\sqrt[3]{x}}\cos\zav{\frac{\sqrt{3}3}{2}\sqrt[3]{x}}}=\sum\limits_{k=0}^{\infty}\frac{3^{3k}x^k}{(3k)!}$:}
f\zav{\hypzav{\nadsebou{\frac13\quad \frac23}{c\quad d}}x}&=: \!\! \ _0 F_2\zav{\nadsebou{-}{c\quad d};x}.
\end{align}
In fact, \textit{any} generalized hypergeometric function $\!\! \ _p F_q$ for $p\leq q+1 $ can be constructed this way. Since the hypergeometrization does not change the region of convergence, we can see at once from this construction that the series $\!\! \ _{q+1} F_q$ converges in the unit disk (since those functions originated from $(1-x)^{-b}$) and the rest $\!\! \ _{p} F_q$ $(p\leq q)$ converges everywhere since they are constructed from entire functions like $e^x,\cosh(2\sqrt{x})$ etc.

Similarly, we can represent some multi-variable hypergeometic function like Appell's function: 
\begin{align}
\intertext{Appell's $F_1$ function from $f(t):=(1-x t)^{-b_1}(1-y t)^{-b_2}$:}
 f\zav{\hypzav{\nadsebou{a}{c}}t}&=:F_1 \zav{\nadsebou{a}{c};\nadsebou{b_1\quad b_2}{-};xt,yt}=\sum_{j,k=0}^{\infty}\frac{(a)_{j+k}}{(c)_{j+k}}\frac{(b_1)_j(b_2)_k}{j!k!}t^{j+k}x^j y^k.\label{F1}\\
\intertext{Appell's $F_2$ function from $f(x):=g_x\zav{\hypzav{\nadsebou{b_2}{c_2}}y},\ g_x(y):=(1-x-y)^{-a}$:}
f\zav{\hypzav{\nadsebou{b_1}{c_1}}x}&:= F_2\zav{\nadsebou{a}{-};\nadsebou{b_1}{c_1}\nadsebou{b_2}{c_2};x,y}=\sum_{j,k=0}^{\infty}\frac{(a)_{j+k}}{j!k!}\frac{(b_1)_j(b_2)_k}{(c_1)_j(c_2)_k}x^j y^k. \label{F2}
\end{align} 
And so on.

Once again, we can retrieve the information about the regions of convergence for Appell's series from their elementary origins. Since the hypergeometrization does not change the radius of convergence, we can deduce from the fact that
$$
(1-x)^{-b_1}(1-y)^{-b_2}=\sum_{j,k=0}^{\infty}\frac{(b_1)_j(b_2)_k}{j!k!}x^j y ^k<\infty \qquad \Leftarrow \qquad \abs{x}<1,\abs{y}<1,
$$
that the same is true for $F_1$ function.

A similar argument lead us to the conclusion that the series for $F_2$ converges in the region $\abs{x+y}<1$. 

This trick is, essentially, \textit{Horn's principle} in reverse. (Horn's principle states that the region of convergence of any hypergeometric function does not depend on the specific values of parameters -- safe for some exceptional pathological values, like negative integers and so on.)

For what follows, we are going to define the function $\tilde F_1$ which is not a member of Appell's family not even appears on Horn's list but it can be thought of as a generalization of $F_1$ function:
\begin{definition} Let 
$$
f(t):=\!\! \ _2 F_1\zav{\nadsebou{a_1\quad a_2}{c};x t}(1-y t)^{-b}.
$$
Then
\begin{equation}\label{tildeF1}
f\zav{\hypzav{\nadsebou{\alpha_1\quad \alpha_2}{\gamma_1\quad \gamma_2}}t}=:\tilde F_1\zav{\nadsebou{\alpha_1\quad \alpha_2}{\gamma_1\quad \gamma_2};\nadsebou{a_1\quad a_2}{c}\nadsebou{b}{-};tx,ty}=\sum_{j,k=0}^{\infty}\frac{(\alpha_1)_{j+k}(\alpha_2)_{j+k}}{(\gamma_1)_{j+k}(\gamma_2)_{j+k}}\frac{(a_1)_j(a_2)_j}{(c)_j j!}\frac{(b)_k}{k!}t^{j+k}x^jy^k.
\end{equation}
\end{definition}
Since the Taylor series $f(t)$ from above converges for $\abs{xt}<1,\abs{yt}<1$ the same can be said about $\tilde F_1$.

Later we are going to show that the integral $I_\alpha$ can be written in terms of $\tilde F_1$.

\subsection{Properties}
Hypergeometrization appears naturally in multiple settings. It can used to described the remainder of a function in Taylor expansion:
\begin{proposition}
\begin{equation}\label{Taylrem}
f(x)=\sum_{k=0}^{n-1}f^{(k)}(t)\frac{x^k}{k!}+\frac{x^n}{n!} f^{(n)}\zav{\hypzav{\nadsebou{1}{n+1}}x}.
\end{equation}
\end{proposition} 
It can be used to handle differentiation
\begin{proposition}
$$
\partial_x x^\beta f\zav{x^\alpha}=\beta x^{\beta-1} f\zav{\hypzav{\nadsebou{1+\frac{\beta}{\alpha}}{\frac{\beta}{\alpha}}} x^\alpha},\qquad \beta\not=0,
$$
\end{proposition}
and for $\alpha=1$ even repeated differentiation:
\begin{proposition}
$$
\frac{\partial^n_x}{n!} x^\beta f\zav{ x}=\zav{\nadsebou{\beta}{n}} x^{\beta-n} f\zav{\hypzav{\nadsebou{1+\beta}{1+\beta-n}} x}.
$$
\end{proposition}
We can also represent the result of various definite integrals. First:
\begin{proposition}
For $\Re(c)>\Re(a)>0$ it holds
$$
\inte{0}{1}s^{a-1}(1-s)^{c-a-1}f(s x){\rm d}s=\frac{\Gamma(c-a)\Gamma(a)}{\Gamma(c)}f\zav{\hypzav{\nadsebou{a}{c}} x}.
$$
\end{proposition}
And second:
\begin{proposition}\label{intrep2}
For $\Re(a)>0,\Re(c)>0,\Re(\alpha)>0,$ it holds
$$
\inte{0}{1}s^{a-1}(1-s^\alpha)^{c-1}f\zav{x s^\alpha(1-s^\alpha)}{\rm d}s=\frac{\Gamma\zav{\frac{a}{\alpha}}\Gamma(c)}{\alpha\Gamma\zav{c+\frac{a}{\alpha}}} f\zav{\hypzav{\nadsebou{\frac{a}{\alpha}}{\frac{c}{2}+\frac{a}{2\alpha}} \nadsebou{c}{\frac{1+c}{2}+\frac{a}{2\alpha}}}\frac{x}{4}}.
$$
\end{proposition}
All of these facts are easy to derive.

But most importantly for our goal, it can be used to describe antiderivatives as shown in Proposition \ref{HFFTC}, which we are going to prove.
\begin{proof}
Proposition follows from the fact that
$$
\int x^{\alpha+\beta k}{\rm d}x=\frac{x^{\alpha+\beta k+1}}{\alpha+\beta k+1}=\frac{x^{\alpha+1}}{\alpha+1}\frac{\zav{\frac{\alpha+1}{\beta}}_k}{\zav{1+\frac{\alpha+1}{\beta}}_k }x^{\beta k}+c,\qquad k=0,1,2,\dots
$$
\end{proof} 
Thus, to compute the indefinite integral (\ref{HFFTCformula}) is equivalent to adding two Pochhammer symbols into the Taylor series expansion of $f$, one in the numerator and one in the denominator. These two symbols in addition differs by one.

In the rest of the paper we are going to explore how this is in any way helpful in various special cases of $f$.
\begin{remark}
The concept of hypergeometrization was introduced by the present author in \cite{B2}. It can be also understand as a Hadamard product (or convolution)
$$
 f\zav{\hypzav{\nadsebou{a}{c}}x}=\!\! \ _{2} F_1\zav{\nadsebou{a \quad 1}{c};x}\star f(x), 
$$
where the Hadamard product of the two formal power series $g(x)=\sum_{k\geq 0} g_k x^k$, $h(x)=\sum_{k\geq 0} h_k x^k$ is defined 
$$
g(x)\star h(x):=\sum_{k=0}^{\infty}g_k h_k x^k.
$$
Before \cite{B2}, a linear operator which brings a function to its Hadamard product with some hypergeometric function (i.e. to its hypergeometrization) appeared also in \cite{Carlson} and elsewhere. But hypergeometrization is a special case of Hadamard product, and has many properties the general Hadamard product does not posses.
\end{remark}
\section{Generalized hypergeometric functions}\label{S3}
The easiest case is when the Taylor series of $f$ contains nothing but Pochhammer symbols -- these functions are called ``generalized hypergeometric function'' $\!\! \ _p F_q$ defined for $p\leq q+1$ by the infinite series:
\begin{equation}\label{seriesdef}
\!\! \ _p F_q\zav{\nadsebou{a_1 \dots a_p}{c_1 \dots c_q};x}:=\sum_{k=0}^{\infty}\frac{(a_1)_k\cdots (a_p)_k}{(c_1)_k\cdots (c_q)_k}\frac{x^k}{k!},\qquad 1-c_i\not\in \mathbb{N}, \forall i.
\end{equation}
For this family of function Proposition \ref{HFFTC} (with  $f=\!\! \ _p F_q$) can be written in the following form: 

\begin{corollary}\label{prophyp} For $p\leq q+1$, $\alpha\not=-1,\beta\not=0$, $-\frac{\alpha+1}{\beta}\not \in \mathbb{N}$ it holds:
$$
\int x^\alpha \!\! \ _p F_q\zav{\nadsebou{a_1 \dots a_p}{c_1 \dots c_q};\gamma x^\beta}{\rm d}x=\frac{x^{\alpha+1}}{\alpha+1} \!\! \ _{p+1} F_{q+1}\zav{\nadsebou{a_1 \dots a_p\quad \frac{\alpha+1}{\beta}}{c_1 \dots c_q\quad 1+\frac{\alpha+1}{\beta}};\gamma x^\beta}+c.
$$
\end{corollary}
In other words, one cannot escape hypergeometric functions making antiderivatives.
This is very convenient since the sort of trouble we encounter with the logarithmic integrals cannot happened now.

Many elementary functions are actually special cases of hypergeometric function, the most simple examples (by number of parameters) are
\begin{align*}
\!\! \ _0 F_0\zav{\nadsebou{-}{-};x}&=e^x,\\
\!\! \ _1 F_0\zav{\nadsebou{a}{-};x}&=(1-x)^{-a},\\
\!\! \ _0 F_1\zav{\nadsebou{-}{\frac12};-\frac{x^2}{4}}&=\cos x,\\
\!\! \ _0 F_1\zav{\nadsebou{-}{\frac32};-\frac{x^2}{4}}&=\frac{\sin x}{x},\\
\!\! \ _1 F_2\zav{\nadsebou{1}{2\quad \frac32};-x^2}&=\zav{\frac{\sin x}{x}}^2,\\
\!\! \ _2 F_1\zav{\nadsebou{ a \quad 1-a}{\frac{1}{2}};-z^2}&=\frac{\zav{\sqrt{1+z^2}+z}^{2a-1}+\zav{\sqrt{1+z^2}-z}^{2a-1}}{2\sqrt{1+z^2}},\\
\!\! \ _2 F_1\zav{\nadsebou{\frac{1}{2}-s\quad 1-s}{\frac32};-t^2}&=\frac{\sin\zav{2s\ {\rm arctan }(t)}}{2st}(1+t^2)^{s},
\end{align*}
and the list goes on and on (see \cite{dlmf15.4}). In fact, sheer size of this list makes it difficult to just memorize it. Luckily, great number of other elementary function can be converted into its hypergeometric form using Corollary \ref{prophyp}! 

For example, since  $\sin x = \int \cos x{\rm d}x$ we have
$$
\sin x=\int \cos x{\rm d}x=\int\!\! \ _0 F_1\zav{\nadsebou{-}{\frac12};-x^2}{\rm d}x= x\!\! \ _1 F_2\zav{\nadsebou{\frac12}{\frac12\quad \frac32};-x^2}=x\!\! \ _0 F_1\zav{\nadsebou{-}{\frac32};-x^2}.
$$

Similarly
\begin{align*}
\int e^x{\rm d}x&=x\!\! \ _1 F_1\zav{\nadsebou{1}{2};x} & &\Rightarrow & \frac{e^x-1}{x}&=\!\! \ _1 F_1\zav{\nadsebou{1}{2};x}.\\
\int \frac{1}{1-x}{\rm d}x&=\int \!\! \ _1 F_0\zav{\nadsebou{1}{-};x}{\rm d}x=x\!\! \ _2 F_1\zav{\nadsebou{1\quad 1}{2};x} & &\Rightarrow & \ln\frac{1}{1-x}&=x\!\! \ _2 F_1\zav{\nadsebou{1\quad 1}{2};x}.\\
\int \frac{1}{1+x^2}{\rm d}x&=\int \!\! \ _1 F_0\zav{\nadsebou{1}{-};-x^2}{\rm d}x=x\!\! \ _2 F_1\zav{\nadsebou{1\quad \frac12}{\frac32};-x^2}& &\Rightarrow &{\rm arctan}(x)&=x\!\! \ _2 F_1\zav{\nadsebou{1\quad \frac12}{\frac32};-x^2}.\\
\int \frac{1}{\sqrt{1-x^2}}{\rm d}x&=\int \!\! \ _1 F_0\zav{\nadsebou{\frac12}{-};x^2}{\rm d}x=x\!\! \ _2 F_1\zav{\nadsebou{\frac12\quad \frac12}{\frac32};x^2}& &\Rightarrow &{\rm arcsin}(x)&=x\!\! \ _2 F_1\zav{\nadsebou{\frac12\quad \frac12}{\frac32};x^2}.\\
\end{align*}

Using the notation of hypergeometric functions might seem cumbersome and even inelegant at the first glance, but its added benefit is that much more information about the function is readily available. For example the notation
$$
x\!\! \ _2 F_1\zav{\nadsebou{1\quad \frac12}{\frac32};-x^2},
$$
tells us that we are looking at an odd function and also gives us the exact prescription for the $n$-term of its Taylor series. While  the symbol
$$
\arctan(x),
$$
points us only to the fact that this happens to be an inverse of something (specifically, tan$(x)$ function).

Another benefit is that we can describe antiderivatives, which are impossible to write using elementary functions only. In fact, many famous special functions can be written in terms of hypergeometric function. For instance:
\begin{align*}
\intertext{Error function:}
\int e^{x^2}{\rm d}x&=x\!\! \ _1 F_1\zav{\nadsebou{\frac12}{\frac32};x^2} & &\Rightarrow & {\rm erf}(x)&=\frac{2}{\sqrt{\pi}}\!\! \ _1 F_1\zav{\nadsebou{\frac12}{\frac32};x^2}.\\
\intertext{Sine integral:}
\int \frac{\sin x}{x}{\rm d}x&= \int \!\! \ _0 F_1\zav{\nadsebou{-}{\frac32};-\frac{x^2}{4}}{\rm d}x=x\!\! \ _1 F_2\zav{\nadsebou{\frac12}{\frac32\quad \frac32};-\frac{x^2}{4}} & &\Rightarrow & {\rm Si}(x)&=x\!\! \ _1 F_2\zav{\nadsebou{\frac12}{\frac32\quad \frac32};-\frac{x^2}{4}}.\\
\intertext{Inverse tangent integral:}
\int \frac{\arctan x}{x}{\rm d}x&= \int \!\! \ _2 F_1\zav{\nadsebou{1\quad \frac12}{\frac32};-x^2}{\rm d}x=x\!\! \ _3 F_2\zav{\nadsebou{1\quad \frac12\quad \frac12}{\frac32\quad \frac32};-x^2} & &\Rightarrow & {\rm Ti}(x)&=x\!\! \ _1 F_2\zav{\nadsebou{\frac12}{\frac32\quad \frac32};-\frac{x^2}{4}}.\\
\end{align*}
To list a few.

Hypergeometric functions can be also used to express solution to a ``trinomial'' equation:
$$
\alpha x^n+x=a,
$$
as
\begin{equation}\label{trieq}
x=a \!\!\ _{n} F_{n-1}\zav{\nadsebou{\frac{1}{n}\quad \dots \frac{n-1}{n}\quad 1 }{\frac{2}{n-1}\quad \frac{3}{n-1} \dots \frac{n+1}{n-1}};-\frac{\alpha n^{n} a^{n-1}}{(n-1)^{n-1}}}.
\end{equation}
And even some number-theoretical function:
\begin{equation}\label{zetaeta}
\zeta(k)=\!\!\ _{k+1} F_{k}\zav{\nadsebou{1\quad 1 \ \dots\ 1 \quad 1}{2\quad 2 \ \dots\ 2};1},\qquad
\eta(k)=\!\!\ _{k+1} F_{k}\zav{\nadsebou{1\quad 1 \ \dots\ 1 \quad 1}{2\quad 2 \ \dots\ 2};-1}.
\end{equation}

\subsection{Value at $x=0$}
Since it is possible to represent many antiderivatives in terms of hypergeometric function using Corollary \ref{prophyp}, we can compute definite integrals ``freshmen style'', i.e. using Fundamental theorem of calculus.

For this we will need to known the values of hypergeometric functions at specific points. The amazing thing is that many of these values are, indeed, known and there are quite general formulas valid for (almost) arbitrary parameters in some cases.  

The most trivial case is $x=0$. For this we have
\begin{equation}\label{value0}
\!\! \ _p F_q\zav{\nadsebou{a_1\dots a_p}{c_1\dots c_q};0} = 1.
\end{equation}
A fact that follows immediately from the definition of $\!\! \ _p F_q$ for $p+1\leq q$. 
\subsection{Value at $x=-\infty$}
The next easiest value to understand is at $x=-\infty$. There might be no value as such and the asymptotic behavior is, generally, quite complicated but it can be simplified to a remarkable degree for the purposes of integration, as follows:
\begin{lemma}
Assuming $a_j-a_k\not \in \mathbb{Z}$ $\forall j,k$, $j\not=k$, $1-a_j\not\in\mathbb{N}$, $\forall j$ provided $p\leq q+1$ it holds:
\begin{equation}\label{valueinfty}
(-x)^{\alpha}\!\! \ _p F_q\zav{\nadsebou{a_1 \dots a_p}{c_1 \dots c_q};x}\qquad \longrightarrow\qquad \prod_{{\tiny\nadsebou{i=1}{a_i\not=\alpha}}}^{p}\frac{\Gamma(a_i-\alpha)}{\Gamma(a_i)}\prod_{j=1}^{q}\frac{\Gamma(c_j)}{\Gamma(c_j-\alpha)},\qquad (x\to -\infty),
\end{equation}
\textit{if and only if $p\geq q-1$ and}
\begin{align*}
\text{for } p&>q-1 & \alpha&={\rm min}(a_1,\dots, a_p),\\
\intertext{or }
\text{for } p&=q-1 & \alpha&={\rm min}(a_1,\dots, a_p)<\sigma-\frac12, & \sigma&:=\sum_j c_j-\sum_i a_i.
\end{align*}
\end{lemma}
\begin{proof}
The proof is break down into multiple cases. 

\bigskip
\emph{Case:} $p=q+1$. Define a \textit{standard term} $ST_{a}(x)$ associated to the upper parameter $a$ as follows:
$$ ST_{a}(x):=\prod_{j=1}^{p-1}\frac{\Gamma(a_j-a)}{\Gamma(a_j)}\prod_{i=1}^{q}\frac{\Gamma(c_i)}{\Gamma(c_i-a)} (-x)^{-a} \!\! \ _{q+1} F_{p-1}\zav{\nadsebou{a\quad 1-c_1+a\dots 1-c_q+a}{1-a_1+a\dots 1-a_{p-1}+a};(-1)^{p+1-q}\frac{1}{x}}.
$$
The asymptotic behavior for a large argument of $\!\!  _{q+1} F_q$ is govern by standard terms only. Specifically it holds:  

For $\abs{\arg(-x)}< \pi$ we have 
\begin{eqnarray}\label{AC8}
\!\! \ _{q+1} F_{q}\zav{\nadsebou{  a_1 \quad \dots \quad a_{q+1}}{ c_1\quad \dots \quad c_q};x}&=&\sum_{j=1}^{q+1} ST_{a_j}(x).
\end{eqnarray}
See \cite[16.11.6]{dlmfeq}. Since the principal behavior of any standard term $ST_a(x)$ is polynomial growth 
$$
ST_a(x)\propto (-x)^{-a},\qquad (x\to -\infty),
$$
the dominant behavior would arrive from the lowest upper parameter. Hence, we arrive at the result.

\bigskip
\emph{Case: }$p=q$. Asymptotic behavior of $\!\! \ _q F_q$ is no longer govern by standard terms only, there is in addition an ``exponential term'':
 
For $\abs{\arg(-x)}< \pi$ we have
\begin{eqnarray}
\!\! \ _{q} F_{q}\zav{\nadsebou{  a_1 \quad \dots \quad a_{q}}{ c_1\quad \dots \quad c_q};x}&\sim &\sum_{j=1}^{q} ST_{a_j}(x)+\frac{\Gamma(c_1)\cdots \Gamma(c_q)}{\Gamma(a_1)\cdots \Gamma(a_q)} e^{x}x^{-\sigma},\qquad (\abs{x}\to \infty).
\end{eqnarray} 
See \cite[16.11.7]{dlmfeq}. Since the exponential term is negligible as $x\to -\infty$, the result follows as well.

\bigskip
\emph{Case: }$p=q-1$. In this case we have two exponential terms, one for every square root of the argument:

\begin{eqnarray}
\!\! \ _{q-1} F_{q}\zav{\nadsebou{  a_1 \quad \dots \quad a_{q-1}}{ c_1\quad \dots \quad c_q};-\frac{x^2}{4}}&\sim &\sum_{j=1}^{q-1} ST_{a_j}\zav{-\frac{x^2}{4}}\\
&+& \gamma e^{\imag x}\zav{\frac{\imag x}{2}}^{\frac12-\sigma}\\
&+& \gamma e^{-\imag x}\zav{-\frac{\imag x}{2}}^{\frac12-\sigma},\qquad (x\to \infty),
\end{eqnarray}
where $\gamma$ is some real constant depending on the parameters. See \cite[16.11.8]{dlmfeq}. 

In this case the exponential terms behaves only polynomially (because of imaginary unit present in the exponentials). Hence to ensure the dominance of the standard term associated to the least upper parameter, it is further necessary to make sure that this minimum parameter $a$ is strictly lower than $\sigma-\frac12$.

\bigskip
\emph{Case: }$p<q-1$. In this case there are so many exponential terms (one for every $q+1-p$-th root of the argument) that it is impossible for $\!\! \ _p F_q$ to growth only polynomially (see \cite[16.11.9]{dlmfeq}). One exponential is always dominant as $x\to -\infty$. Thus the limit (\ref{valueinfty}) cannot exists.
\end{proof}

The case when one of the upper parameters is a non-positive integer is even simpler:
\begin{lemma}
Assuming $n+1\in\mathbb{N}$, $a_j-a_k\not \in \mathbb{Z}$ $\forall j,k$, $j\not=k$, $1-a_j\not\in\mathbb{N}$, $\forall j$ provided $p\leq q+1$ it holds:
\begin{equation}\label{valueinftypol}
(-x)^{-n}\!\! \ _p F_q\zav{\nadsebou{-n\quad a_2 \dots a_p}{c_1 \dots c_q};x}\qquad \longrightarrow\qquad \prod_{i=2}^{p}(a_i)_n\prod_{j=1}^{q}\frac{1}{(c_j)_n},\qquad (x\to -\infty).
\end{equation}
\end{lemma}
\begin{proof} The $\!\! \ _p F_q$ function is in fact a polynomial, which is actually \textit{equal} to its standard term associated with $-n$ parameter:
$$
\!\! \ _p F_q\zav{\nadsebou{-n\quad a_2\dots a_p}{c_1\dots c_q};x}=ST_{-n}(x).
$$   
See \cite[16.2.3]{dlmfeq2}
\end{proof}
\begin{example}
As we saw the arctan$(x)$ function can be described using hypergeometric functions as
$$
\int \frac{1}{1+x^2}{\rm d}x=\int \frac\!\! \ _1 F_0\zav{\nadsebou{1}{-};-x^2}{\rm d}x\stackrel{(\ref{HFFTCformula})}{=} x\!\!\ _2 F_1\zav{\nadsebou{1\quad \frac12}{\frac32};-x^2}.
$$
By (\ref{value0}) we thus have arctan$(0)=0$ and since the smallest upper parameter equals to $\frac12$ we have 
$$
\inte{0}{\infty}\frac{1}{1+x^2}{\rm d}x=\hzav{x\!\!\ _2 F_1\zav{\nadsebou{1\quad \frac12}{\frac32};-x^2}}^{\infty}_{0}\stackrel{(\ref{valueinfty}),(\ref{value0})}{=}\frac{\Gamma(\frac32)\Gamma(\frac12)}{\Gamma(1)\Gamma(1)}-0=\frac{1}{2}\Gamma^2\zav{\frac12}.
$$
Since we also know that arctan$(\infty)=\frac{\pi}{2}$ from the fact that it is an inverse of tan$(x)$, we can conclude that
$$
\Gamma^2\zav{\frac12}=\pi.
$$
This is a much easier proof of this very well known value of Gamma function than the standard one computing
$$
\inte{-\infty}{\infty}e^{x^2}=\sqrt{\pi},
$$
using multi-variable calculus!
\end{example}
\begin{example}
Similarly we can compute
$$
\inte{0}{\infty}\frac{1}{1+x^3}{\rm d}x=\inte{0}{\infty}\!\! \ _1 F_0\zav{\nadsebou{1}{-};-x^3}{\rm d}x\stackrel{(\ref{HFFTCformula})}{=} \hzav{x\!\! \ _2 F_1\zav{\nadsebou{1\quad \frac13}{\frac43};-x^3}}^{\infty}_{0}\stackrel{(\ref{valueinfty})}{=}\frac{\Gamma(\frac23)\Gamma(\frac43)}{\Gamma(1)\Gamma(1)}
$$
$$
=\frac{\Gamma\zav{\frac23}\Gamma\zav{\frac13}}{3}=\frac{\pi}{3\sin\frac{\pi}{3}}=\frac{2\pi}{3\sqrt{3}}.
$$
\end{example}
\begin{example}
Since we can represent solutions of the trinomial equation by (\ref{trieq}), we can even solve problems of the following type:
$$
\inte{0}{\infty}\frac{y}{x^{\frac75}}{\rm d}x,\qquad  \alpha y^5+y=x.
$$

Since by (\ref{trieq})
$$
y=x\!\! \ _4 F_3\zav{\nadsebou{\frac15\quad \frac25\quad \frac35\quad \frac45}{ \frac24\quad \frac34\quad \frac54};-\alpha\frac{5^5}{4^4}x^{4}},
$$
we have
$$
\inte{0}{\infty}\frac{y}{x^{\frac75}}{\rm d}x=\inte{0}{\infty}x^{-\frac{2}{5}}\!\! \ _4 F_3\zav{\nadsebou{\frac15\quad \frac25\quad \frac35\quad \frac45}{ \frac24\quad \frac34\quad \frac54};-\alpha\frac{5^5}{4^4}x^{4}} {\rm d}x
=\hzav{\frac{5x^{\frac35}}{3}\!\! \ _5 F_4\zav{\nadsebou{\frac{3}{20}\quad\frac15\quad \frac25\quad \frac35\quad \frac45}{\frac{23}{20}\quad \frac24\quad \frac34\quad \frac54};-\alpha\frac{5^5}{4^4}x^{4}} }^{\infty}_{0}
$$
$$
\stackrel{(\ref{valueinfty})}{=}\frac{5}{3}\zav{\alpha\frac{5^5}{4^4}}^{-\frac{3}{20}}\frac{\Gamma(\frac{1}{20})\Gamma(\frac{5}{20})\Gamma(\frac{9}{20})\Gamma(\frac{13}{20})\Gamma(\frac{2}{4})\Gamma(\frac{3}{4})\Gamma(\frac{5}{4})\Gamma(\frac{23}{20})}{\Gamma(\frac{1}{5})\Gamma(\frac{2}{5})\Gamma(\frac{3}{5})\Gamma(\frac{4}{5})\Gamma(\frac{7}{20})\Gamma(\frac{13}{20})\Gamma(\frac{22}{20})}=\alpha^{-\frac{3}{20}}\frac{\Gamma(\frac{3}{20})\Gamma(\frac{1}{4})}{4\Gamma(\frac{7}{5})}.
$$
Generally:
$$
\inte{0}{\infty}x^\beta y{\rm d}x=\alpha^{-\frac{\beta+2}{4}}\frac{\Gamma(\frac{\beta+2}{4})\Gamma(-\frac54 \beta-\frac32)}{4 \Gamma(-\beta)},\qquad -2<\beta<-\frac65.
$$
The constrains on beta stems from the fact that $\beta+2$ (the added parameter from integration) must be the positive and smallest upper parameter in order the limit in $\ref{valueinfty}$ to exists.
\end{example}
\begin{example} We are now ready to prove the identity
$$
\inte{0}{\infty} \frac{\sqrt{\sqrt{1+x}-1}}{x^{\frac{11}{8}}}{\rm d}x=\frac{4 \Gamma^2\zav{\frac14}}{3\sqrt{2-\sqrt{2}}\sqrt{\pi}}.
$$
The hardest step with our method is always to represent the integrand as a hypergeometric function (if it is possible). Here (as we can see) it does not appear to be at all obvious how to proceed. Especially, the nested roots are seemingly hard to handle. 

The function under the outer root sign can be represented quite easily:
$$
\sqrt{1+x}-1=\frac12\inte{0}{x} \frac{1}{\sqrt{t+1}}{\rm d}t=\frac12 \inte{0}{x}\!\! \ _1 F_0\zav{\nadsebou{\frac12}{-};-t}{\rm d}t\stackrel{(\ref{HFFTCformula})}{=}
\frac12 x \!\! \ _2 F_1\zav{\nadsebou{\frac12\quad 1 }{2};-x}.
$$
So now we just take square root of this hypergeometric function. Now hypergeometric functions are not in general closed under multiplication, the one exception being: 
\begin{align*}
\!\! \ _1 F_0\zav{\nadsebou{a}{-};x}&\!\! \ _1 F_0\zav{\nadsebou{b}{-};x}=\!\! \ _1 F_0\zav{\nadsebou{a+b}{-};x},\\
\!\! \ _1 F_0\zav{\nadsebou{a}{-};x}&\!\! \ _1 F_0\zav{\nadsebou{a}{-};-x}=\!\! \ _1 F_0\zav{\nadsebou{a}{-};x^2}.\\
\end{align*}
For other cases, only for specific values of parameters can be multiplication of two hyp. functions represented by another hyp. function. Fortunately, that is our case since there is a famous formula
\begin{equation}\label{2F12F1}
\!\! \ _2 F^2_1\zav{\nadsebou{ a \quad b}{a+b+\frac{1}{2}};x}=\!\! \ _3 F_2\zav{\nadsebou{ 2a \quad a+b\quad 2b}{a+b+\frac{1}{2}\quad 2a+2b};x}.
\end{equation}

From this we can see that for $a=\frac14,b=\frac34$ it holds
$$
\!\! \ _2 F^2_1\zav{\nadsebou{ \frac14 \quad \frac34}{\frac32};x}=\!\! \ _3 F_2\zav{\nadsebou{ \frac12 \quad 1\quad \frac32}{\frac{3}{2}\quad 2};x}=\!\! \ _2 F_1\zav{\nadsebou{ \frac12 \quad 1}{\frac{3}{2}};x}.
$$
Thus
\begin{equation}\label{sqsqrep}
\sqrt{\frac{\sqrt{1+x}-1}{x}}=\frac{1}{\sqrt{2}} \!\! \ _2 F_1\zav{\nadsebou{\frac14\quad \frac34}{\frac32};-x},
\end{equation}
and hence
$$
\inte{0}{\infty} \frac{\sqrt{\sqrt{1+x}-1}}{x^{\frac{11}{8}}}{\rm d}x=\inte{0}{\infty} x^{-\frac78}\frac{1}{\sqrt{2}} \!\! \ _2 F_1\zav{\nadsebou{\frac14\quad \frac34}{\frac32};-x}{\rm d}x\stackrel{(\ref{HFFTCformula})}{=}\hzav{\frac{8}{\sqrt{2}}x^{\frac18} \!\! \ _3 F_2\zav{\nadsebou{\frac14\quad \frac34\quad \frac18}{\frac32\quad \frac98};-x}}^{\infty}_{0}
$$
$$
\stackrel{(\ref{valueinfty})}{=}\frac{8 }{\sqrt{2}} \frac{\Gamma(\frac18)\Gamma(\frac58)\Gamma(\frac{9}{8})\Gamma(\frac32)}{\Gamma(\frac14)\Gamma(\frac34) \Gamma(1)\Gamma(\frac{11}{8})}=\dots =
\frac{4 \Gamma^2\zav{\frac14}}{3\sqrt{2-\sqrt{2}}\sqrt{\pi}},
$$
where in the dots, various well-known identities valid for the Gamma function must be used to obtain the result. (Specifically: formula for the double argument, reflection formula and $\Gamma(x+1)=x\Gamma(x)$.)
\end{example}
\begin{example}
Of course, armed with the representation (\ref{sqsqrep}) one can easily produce general identity
$$
\inte{0}{\infty} x^{\alpha-1}\sqrt{\sqrt{1+x}-1}{\rm d}x=\sqrt{\frac{2}{\pi}} \Gamma(2\alpha)\Gamma\zav{-\frac12-2\alpha}\sin(\pi \alpha), \qquad -\frac{1}{2}<\alpha<-\frac{1}{4}.
$$
\end{example}
\begin{example}
The representation (\ref{sqsqrep}) can be also seen as a consequence of the identity
$$
\!\! \ _2 F_1\zav{\nadsebou{ a \quad b}{a+b+\frac{1}{2}};x}=\!\! \ _2 F_1\zav{\nadsebou{ 2a \quad 2b}{a+b+\frac{1}{2}};\frac{1-\sqrt{1-x}}{2}}.
$$
From this we can see that (\ref{sqsqrep}) is, in fact, a special case of a general formula 
\begin{equation}\label{gensqrrep}
\zav{\frac{\sqrt{1+x}-1}{x}}^{\beta}=2^{-\beta}\!\! \ _2 F_1\zav{\nadsebou{\frac{\beta}{2}\quad \frac{\beta+1}{2}}{\beta+1};-x}.
\end{equation}
Thus we can compute
$$
\inte{0}{\infty} x^{\alpha-1}\zav{{\sqrt{1+x}-1}}^{\beta}{\rm d}x=\frac{\beta\Gamma(-\beta-2\alpha)\Gamma(\alpha+\beta)2^{2\alpha+\beta}}{ \Gamma(1-\alpha)}, \qquad -\beta<\alpha<-\frac{\beta}{2}.
$$
\end{example}
\begin{example}
As a final touch we can represent many seemingly different function by the same hypergeometric function just noticing that
$$
\zav{\frac{\sqrt{1+x}-1}{x}}^{\beta}=\zav{\frac{\zav{\sqrt{1+x}-1}^n}{x^n}}^{\frac{\beta}{n}},
$$
for $n\in \mathbb{N}$. Setting $n=4$ we obtain the final result (\ref{maplenogo}):
$$
\inte{0}{\infty}\sqrt[8]{\frac{x^2+8x+8-4(2+x)\sqrt{1+x}}{x^{11}}}{\rm d}x=\frac{4 \Gamma^2\zav{\frac14}}{3\sqrt{2-\sqrt{2}}\sqrt{\pi}}.
$$
\end{example}

\subsection{Summation formulas}
Since the operation of taking the antiderivative is invariant in the generalized hypergeometric function family, the task of evaluating definite integrals of functions from this family is effectively reduced to applying one of the many known ``summation formulas'', i.e. values of hypergeometric functions at specific points (besides $0$ and $-\infty$).

Out of these, the most famous is Gauss's summation formula:
\begin{equation}\label{2F11}
\!\! \ _2 F_1\zav{\nadsebou{a\quad b}{c};1}=\frac{\Gamma(c)\Gamma(c-a-b)}{\Gamma(c-a)\Gamma(c-b)},\qquad c>a+b,
\end{equation}
which is (to the authors knowledge) the only one for which the number of ``free'' parameters is maximal.

Formulas for other arguments than $x=1$ are usually derived using some transform that bring a desired point to $x=1$. For example, from the identity
$$
\!\! \ _2 F_1\zav{\nadsebou{ a \quad b}{1+a-b};x}=(1-x)^{-a}\!\! \ _2 F_1\zav{\nadsebou{ \frac{a}{2} \quad \frac{1+a}{2}-b}{1+a-b};-\frac{4x}{(1-x)^2}} ,
$$
one can obtain putting $x=-1$ and using (\ref{2F11}):
\begin{equation}\label{2F1-1}
\!\! \ _2 F_1\zav{\nadsebou{ a \quad b}{1+a-b};-1}=2^{-a}\frac{\Gamma(1+a-b)\Gamma(\frac12)}{\Gamma\zav{1+\frac{a}{2}-b}\Gamma(\frac{1+a}{2}-b)}.
\end{equation}
Applying the Pfaff transform:
\begin{equation}\label{Pfaff}
\!\! \ _2 F_1\zav{\nadsebou{ a \quad b}{c};x}=(1-x)^{-b}\!\! \ _2 F_1\zav{\nadsebou{ c-a \quad b}{c};\frac{x}{x-1}},
\end{equation}
on the left hand side we can obtain a summation formula for $x=\frac12$:
\begin{equation}\label{2F11/2}
\!\! \ _2 F_1\zav{\nadsebou{ a \quad 1+a-2b}{1+a-b};\frac12}=\frac{\Gamma(1+a-b)\Gamma(\frac12)}{\Gamma\zav{1+\frac{a}{2}-b}\Gamma(\frac{1+a}{2}-b)}.
\end{equation}
And so on. 

The cost of this procedure is usually one (or more) free parameter.

Actually, since the Pfaff transform changes the argument from $x=1$ to $-\infty$, Gauss's summation formula (\ref{2F11}) itself, can be derived this way. 
The argument is as follows: Assuming $a<c-b$, $c-b-a\not\in \mathbb{Z}$ it holds as $x\to -\infty$:
\begin{align}
(-x)^{a}(1-x)^{-a} & \!\! \ _2 F_1\zav{\nadsebou{a\quad b}{c};\frac{x}{x-1}}& &\stackrel{(\ref{Pfaff})}{=}& (-x)^{a}&\!\! \ _2 F_1\zav{\nadsebou{a\quad c-b}{c};x},\nonumber\\
&\downarrow & &\ \ \Downarrow & &\downarrow\quad (\ref{valueinfty}) \nonumber\\
&\!\! \ _2 F_1\zav{\nadsebou{a\quad b}{c};1} & &\ \ =& &\frac{\Gamma(c)\Gamma(c-b-a)}{\Gamma(c-a)\Gamma(c-b)}.\nonumber
\end{align}

The case $c-b-a\in \mathbb{Z}$ can be dealt with using a simple limit argument.

Hence, we can think of the value at $-\infty$ (\ref{valueinfty}) as the fundamental value of a hypergeometric function, from which other values are derived.

\bigskip
Summation formulas are extremely important in many application, especially if its values can be written in terms of Gamma function as was the case in the shown examples.

For instance, if the product of two hypergeometric functions can be written once again as a hypergeometric function, is closely related to a (terminating) summation formula. Take the product we encounter before (\ref{2F12F1}):
$$
\!\! \ _2 F^2_1\zav{\nadsebou{ a \quad b}{a+b+\frac{1}{2}};x}=\!\! \ _3 F_2\zav{\nadsebou{ 2a \quad a+b\quad 2b}{a+b+\frac{1}{2}\quad 2a+2b};x}.
$$
Rearranging the double sum on the left hand side by $j\to j-k$,
$$
(LHS)=
\sum_{j,k=0}^{\infty}\frac{(a)_k(b)_k}{(a+b+\frac12)_k k!}\frac{(a)_j(b)_j}{(a+b+\frac12)_j j!}x^{k+j}=\sum_{j=0}^{\infty}\frac{(a)_{j-k}(b)_{j-k}}{(a+b+\frac12)_{j-k} (j-k)!}x^{j}\frac{(a)_k(b)_k}{(a+b+\frac12)_k k!},
$$
using the facts 
$$
(a)_{j-k}=\frac{(a)_j (-1)^k}{(1-a-j)_k}m\qquad \frac{1}{(j-k)!}=\frac{(-j)_k (-1)^k}{j!},
$$
we obtain
$$
=\sum_{j=0}^{\infty}\frac{(a)_{j}(b)_{j}}{(a+b+\frac12)_{j} (j)!}x^{j}\sum_{k=0}^{\infty}\frac{(a)_k(b)_k (-j)_k(\frac12-a-b-j)_k}{(a+b+\frac12)_k(1-a-j)_k(1-b-j)_k k!}
$$
$$
=
\sum_{j=0}^{\infty} \frac{(a)_{j}(b)_{j}}{(a+b+\frac12)_{j} (j)!}x^{j}
\!\! \ _4 F_3\zav{\nadsebou{-j\quad a\quad b\quad \frac12-a-b-j}{1-a-j\quad 1-b-j\quad a+b+\frac12};1}.
$$
Comparing with the like powers of $x$ on the right hand side of (\ref{2F12F1}) we discover terminating version of (essentially) Rogers-Dougall very well posed sum \cite[16.4.9]{dlmf16.4}:
$$
\!\! \ _4 F_3\zav{\nadsebou{-j\quad a\quad b\quad \frac12-a-b-j}{1-a-j\quad 1-b-j\quad a+b+\frac12};1}=\frac{(2a)_j(a+b)_j(2b)_k}{(2a+2b)_j(a)_j(b)_j}.
$$
Similarly from the fact that 
$$
(1-x)^{-a}(1-x)^{-b}=(1-x)^{-a-b},
$$
we can retrieve the terminating version of the Gauss's summation formula (\ref{2F11}) and from
$$
(1-x)^{-a}(1+x)^{-a}=(1-x^2)^{-a},
$$
the terminating version of (\ref{2F1-1}).

Reversing the logic, the \textit{consequence} of (\ref{2F11}) is the formula
$$
\!\! \ _0 F_1\zav{\nadsebou{ -}{a};x}\!\! \ _0 F_1\zav{\nadsebou{ -}{b};x} =\!\! \ _2 F_3\zav{\nadsebou{ \frac{b+a-1}{2}\quad \frac{b+a}{2}}{a\quad b\quad b+a-1};4 x}.
$$

\bigskip
Summation formulas are also of use in number theory since many interesting numbers can be written as a values of hypergeometric functions. For instance
\begin{align}
\zeta(3)&=\!\! \ _4 F_3\zav{\nadsebou{1\quad 1\quad 1\quad 1 }{2\quad 2\quad 2};1}, & \text{Ap\'ery's constant.}\\
L_1&:=\frac{\Gamma^2(\frac14)}{\sqrt{\pi}}=4\sqrt{2}\!\! \ _2 F_1\zav{\nadsebou{\frac12\quad \frac14}{\frac54};1}, & \text{Lemniscate constant.}\\
G&=\!\! \ _3 F_2\zav{\nadsebou{\frac12\quad \frac12\quad 1}{\frac32 \quad \frac32};-1}, & \text{Catalan's constant.}\label{Catalan}\\
e^{\pi}&=\!\! \ _2 F_1\zav{\nadsebou{\imag\quad -\imag}{\frac12};1}+2\!\! \ _2 F_1\zav{\nadsebou{\frac12+\imag\quad \frac12-\imag}{\frac32};1},&\text{Gelfond's constant.}
\end{align}
In the next section we will show how this representation can be used.

\bigskip
For these applications (and many more) summation formulas were extensively studied by many people (\cite{dlmfref1},\cite{dlmfref2},\cite{dlmfref3},\cite{dlmfref4} to mentioned a few) and it remained an open topic even to this day. The point of this paper is that we can use this vast knowledge acquired in centuries also to evaluate (some) definite integrals.

And vice versa: 
\begin{example}\label{3F2-1/3ex}
The value
\begin{equation}\label{3F2-1/3}
\!\! \ _3F_2\zav{\nadsebou{2\quad \frac34\quad \frac54}{\frac74\quad \frac94};-\frac13},
\end{equation}
is unlikely to be a special case of some known summation formula, for there is no transform for $\!\! \ _3F_2$ that brings the point $x=-\frac13$ to $x=1$. (There is one for $\!\! \ _2F_1$ with $1$ free parameter.)

We can nonetheless compute this value noticing that both lower parameters differ form upper ones by 1. Hence it can be represented as an iterated integral:
$$
\frac{x^5}{5\cdot 3}\!\! \ _3F_2\zav{\nadsebou{2\quad \frac34\quad \frac54}{\frac74\quad \frac94};-x^4}=\int x^4 \!\! \ _2F_1\zav{\nadsebou{2\quad \frac34}{\frac74};-x^4}{\rm d}x=\int x \int \frac{x^2}{(1+x^4)^2}{\rm d}x{\rm d}x,  
$$
where in each step the antiderivative is chosen so that its value at $x=0$ is zero.

Since
$$
\int x \int \frac{x^2}{(1+x^4)^2}{\rm d}x{\rm d}x=
\int x\zav{ \frac14\frac{x^3}{(1+x^4)}+\frac{\sqrt{2}}{16}{\rm arctan}\frac{x\sqrt{2} }{1-x^2}+\frac{\sqrt{2}}{32}\ln\frac{x^2-\sqrt{2}x+1}{x^2+x\sqrt{2}+1}}{\rm d}x
$$
\begin{equation}\label{res}
=\frac{x}{8}+\frac{\sqrt{2}}{32}(x^2-1){\rm arctan}\frac{x\sqrt{2}}{1-x^2}+\frac{\sqrt{2}}{64}(x^2+1)\ln\frac{x^2-x\sqrt{2}+1}{x^2+x\sqrt{2}+1},
\end{equation}
the value of (\ref{3F2-1/3}) is the above expression divided by $\frac{x^5}{15}$ and evaluated at $x=1/\sqrt[4]{3}$. It is moreover easy to see that the result is a transcendental number by Baker's Theorem \cite{Baker}.
\end{example}
%

\section{Hypergeometric function with differentiated parameters}\label{S4}
The hardest step in our method is, for a given function, to find a suitable representation as a hypergeometric function. This si often impossible to do.

For example, the function 
$$
{\rm arcsin}^3 x,
$$ 
probably cannot be expressed as a single variable hypergeometric function. The first power can be, of course:
$$
{\rm arcsin} x=
x\!\! \ _2F_1\zav{\nadsebou{\frac12\quad \frac12}{\frac32};x^2},
$$
and since (\ref{2F12F1}) we have also a formula for the second power
$$
{\rm arcsin}^2 x=x^2\!\! \ _3 F_2\zav{\nadsebou{1\quad 1\quad 1}{2\quad \frac32};x^2},
$$
but a similar representation for the third power is unlikely to exists.

Still, we can express it as a hypergeometric function which is differentiated with respect to its parameters like so:
$$
{\rm arcsin}^3 x=-\frac32  \hzav{\epsilon^2} x\!\! \ _2 F_1\zav{\nadsebou{\frac12-\epsilon\quad \frac12+\epsilon}{\frac32};x^2},
$$

We can also represent the second power:
$$
{\rm arcsin}^2x =\frac12\hzav{\epsilon^2}\!\! \ _2 F_1\zav{\nadsebou{\epsilon\quad \epsilon}{\frac12};x^2}.
$$
In fact, here is a short list of a examples we are able to represent this way:
\begin{proposition}
\begin{align}
\ln\frac{1}{1-x}&=\hzav{\epsilon}\!\! \ _1 F_0\zav{\nadsebou{\epsilon}{-};x},\\
(1-x)^{-a}\ln^k\frac{1}{1-x}&=\hzav{\epsilon^k}\!\! \ _1 F_0\zav{\nadsebou{a+\epsilon}{-};x},\label{lnrep}\\
\ln^k\zav{\frac{1+\sqrt{1-x}}{2}}&=\frac{k!}{(-2)^k}\hzav{\epsilon^k}\!\! \ _2 F_1\zav{\nadsebou{\epsilon\quad \frac12+\epsilon}{1+2\epsilon};x},\\
{\rm arcsin}^{2k}x&=\frac{(2k)!}{(-4)^k}\hzav{\epsilon^{2k}} \!\! \ _2 F_1\zav{\nadsebou{-\epsilon\quad \epsilon}{\frac12};x^2},\\
{\rm arcsin }^{2k+1}\!x&=\frac{(2k+1)!}{(-4)^k}\hzav{\epsilon^{2k}} x \!\! \ _2 F_1\zav{\nadsebou{\frac12+\epsilon\quad \frac12-\epsilon}{\frac32};x^2},\\
\sum_{k=1}^{\infty} H_k \frac{x^k}{k!}&=\hzav{\epsilon}\!\! \ _1 F_1\zav{\nadsebou{1}{1-\epsilon};x},& H_k&:=\sum_{j=1}^{k}\frac{1}{j},\\
{\rm Li}_n\zav{x}&= \hzav{\epsilon^n}\!\! \ _n F_{n-1}\zav{\nadsebou{\epsilon\dots \epsilon}{1\dots 1};x},& {\rm Li}_n\zav{x}&:=\sum_{k=1}^{\infty}\frac{x^k}{k^n}.
\end{align}
\end{proposition}
\begin{proof}
The first two formulas are obvious.

The third stems from (\ref{gensqrrep}):
$$
\!\! \ _2 F_1\zav{\nadsebou{\epsilon\quad \frac12+\epsilon}{1+2\epsilon};x}=\zav{\frac{1+\sqrt{1-x}}{2}}^{-2\epsilon}.
$$

The next two are the result of known identities:
$$
\!\! \ _2 F_1\zav{\nadsebou{-\epsilon\quad \epsilon}{\frac12};\sin^2 z}=\cos(2\epsilon z),
$$
$$
\!\! \ _2 F_1\zav{\nadsebou{\frac12+\epsilon\quad \frac12-\epsilon}{\frac32};\sin^2 z}=\frac{\sin(2\epsilon z)}{2\epsilon \sin z},
$$
see \cite[15.4.12,15.4.16]{dlmf15.4}.

The next to the last formula is due to the representation of harmonic numbers $H_k$:
$$
H_k:=\sum_{j=1}^{k}\frac{1}{j}=\psi(k+1)-\psi(1)=\hzav{\epsilon}\frac{\Gamma(k+1)\Gamma(1-\epsilon)}{\Gamma(k+1-\epsilon)}=\hzav{\epsilon}\frac{k!}{(1-\epsilon)_k}.
$$

And, finally, the last formula can be easily verified directly expanding the right hand side by (\ref{Taylrem}) and noting that
$$
{\rm Li}_n\zav{x}= x\, \!\! \ _{n+1} F_{n}\zav{\nadsebou{1\dots 1}{2\dots 2};x}.
$$
\end{proof}
We are going to demonstrate usefulness of such representations on number of examples.
\begin{example} First we prove that
$$
\inte{0}{1}\zav{\frac{{\rm arcsin}\, x}{x}}^3{\rm d}x=\frac34\pi\inte{0}{1}\zav{2\, {\rm arctanh}\,x-\frac{1}{x}\ln\frac{1}{1-x^2}}{\rm d}x=\frac32\pi\ln 2-\frac{\pi^3}{16}.
$$
Note that the equality of two integrals does not stems from any conceivable change of variable.
\begin{align}
\inte{0}{1}\zav{\frac{{\rm arcsin} x}{x}}^3{\rm d}x&=-\frac32  \hzav{\epsilon^2} \inte{0}{1}x^{-2}\!\! \ _2 F_1\zav{\nadsebou{\frac12-\epsilon\quad \frac12+\epsilon}{\frac32};x^2}{\rm d}x \nonumber\\
&\stackrel{(\ref{HFFTCformula})}{=} \frac32  \hzav{\hzav{\epsilon^2}\frac{1}{x}\!\! \ _3 F_2\zav{\nadsebou{\frac12-\epsilon\quad \frac12+\epsilon\quad -\frac12}{\frac32\quad \frac12};x^2}}^{1}_0 \nonumber\\
&=\frac32  \hzav{\epsilon^2}\!\! \ _3 F_2\zav{\nadsebou{\frac12-\epsilon\quad \frac12+\epsilon\quad -\frac12}{\frac32\quad \frac12};1}.\label{auxf1}
\end{align}
The value at the lower integration bound $x=0$ is, indeed, zero since
$$
\hzav{\epsilon^2}\!\! \ _3 F_2\zav{\nadsebou{\frac12-\epsilon\quad \frac12+\epsilon\quad -\frac12}{\frac32\quad \frac12};x^2}\stackrel{(\ref{Taylrem})}{=}
-x^2\frac23\hzav{\epsilon^2}\zav{\frac14-\epsilon^2}\!\! \ _4 F_3\zav{\nadsebou{\frac32-\epsilon\quad \frac32+\epsilon\quad \frac12\quad 1}{\frac52\quad \frac32\quad 2};x^2}.
$$
Next we apply the formula \cite[16.4.11]{dlmf16.4}:
$$
\!\! \ _3 F_2\zav{\nadsebou{a_1 \quad a_2 \quad a_3}{c_1\quad c_2};1}=\frac{\Gamma(c_2)\Gamma\zav{\sigma}}{\Gamma\zav{\sigma+a_3}\Gamma(c_2-a_3)}
\!\!\ _3 F_2\zav{\nadsebou{a_3 \quad c_1-a_1\quad c_1-a_2}{c_1\quad \sigma+a_3};1},
$$
where $\sigma:=c_1+c_2-a_1-a_2-a_3$ is the parameter excess, to get
$$
\!\! \ _3 F_2\zav{\nadsebou{\frac12-\epsilon\quad \frac12+\epsilon\quad -\frac12}{\frac32\quad \frac12};1}=\frac{\pi}{4}\!\! \ _3 F_2\zav{\nadsebou{\epsilon\quad -\epsilon\quad -\frac12}{\frac12\quad 1};1}\stackrel{(\ref{Taylrem})}{=}\frac{\pi}{4}\zav{1+\epsilon^2 \!\! \ _4 F_3\zav{\nadsebou{\epsilon+1\quad 1-\epsilon\quad \frac12\quad 1}{\frac32\quad 2\quad 2};1}},
$$
thus (\ref{auxf1}) is
\begin{align*}
(\ref{auxf1})&=\frac{3\pi}{8} \!\! \ _4 F_3\zav{\nadsebou{1\quad 1\quad \frac12\quad 1}{\frac32\quad 2\quad 2};1}\stackrel{(\ref{HFFTCformula})}{=}\frac{3\pi}{4}\inte{0}{1}x\!\! \ _3 F_2\zav{\nadsebou{1\quad 1\quad \frac12}{2\quad \frac32};x^2}{\rm d}x
\stackrel{(\ref{HFFTCformula})}{=}\frac{3\pi}{4}\inte{0}{1}\int \!\! \ _2 F_1\zav{\nadsebou{1\quad 1}{2};x^2}{\rm d}x{\rm d}x\\
&=\frac{3\pi}{4}\inte{0}{1}\int \frac{1}{x^2}\ln\frac{1}{1-x^2}{\rm d}x{\rm d}x
=\frac{3\pi}{4}\inte{0}{1}\zav{\zav{1-\frac{1}{x}}\ln\frac{1}{1-x}-\zav{1+\frac{1}{x}}\ln\frac{1}{1+x}}{\rm d}x\\
&=\frac{3\pi}{4}\inte{0}{1}\zav{-\!\! \ _2F_1\zav{\nadsebou{1\quad 1}{2};x}+\!\! \ _2F_1\zav{\nadsebou{1\quad 1}{2};-x}+2\ln 2}{\rm d}x\\
&\stackrel{(\ref{HFFTCformula})}{=}\frac{3\pi}{4}\zav{-\!\! \ _2F_1\zav{\nadsebou{1\quad 1\quad 1}{2\quad 2};1}+\!\! \ _2F_1\zav{\nadsebou{1\quad 1\quad 1}{2\quad 2};-1}+2\ln 2}\stackrel{(\ref{zetaeta})}{=}\frac{3\pi}{4}\zav{1-\zeta(2)+2\ln 2-1+\eta(2)}.
\end{align*}
The result now stems from the known values of $\zeta(2)=\frac{\pi^2}{6}$ and $\eta(2)=\frac{\pi^2}{12}$ which we are just going to derive.
\end{example}
\begin{example}
The value of zeta function $\zeta(2)$ can be computed using the representation
$$
\hzav{\epsilon^2}\!\! \ _2 F_1\zav{\nadsebou{\epsilon\quad -\epsilon}{1};1}\stackrel{(\ref{Taylrem})}{=}\hzav{\epsilon^2}\zav{1-\epsilon^2\!\! \ _3 F_2\zav{\nadsebou{\epsilon+1\quad 1-\epsilon\quad 1}{2\quad 2};1}}=-\!\! \ _3 F_2\zav{\nadsebou{1\quad 1\quad 1}{2\quad 2};1}.
$$
But
$$
\hzav{\epsilon^2}\!\! \ _2 F_1\zav{\nadsebou{\epsilon\quad -\epsilon}{1};1}\stackrel{(\ref{2F11})}{=}\hzav{\epsilon^2}\frac{\Gamma(1)\Gamma(1)}{\Gamma(1-\epsilon)\Gamma(1+\epsilon)}=\hzav{\epsilon^2}\frac{\sin \pi \epsilon}{\pi\epsilon}=-\frac{\pi^2}{6}.
$$
\end{example}
\begin{example}
The value of $\eta(2)$ can be derived number of ways. We are going to use the fact that the odd and even part of a hypergeometric function can be written in terms of hypergeometric functions, specifically:
\begin{align}\label{evenodd}
\!\! \ _p F_q\zav{\nadsebou{a_1\dots a_p}{c_1\dots c_q};x}&=
\!\! \ _{2p} F_{2q+1}\zav{\nadsebou{\frac{a_1}{2}\quad \frac{a_1+1}{2} \dots \frac{a_p}{2}\quad \frac{a_p+1}{2}}{\frac{c_1}{2}\quad \frac{c_1+1}{2}\dots \frac{c_q}{2}\quad \frac{c_q+1}{2}\quad \frac12};4^{p-q-1}x^2}\\
&+\frac{a_1\cdots a_p}{c_1\cdots c_q}x\ \!\! \ _{2p} F_{2q+1}\zav{\nadsebou{\frac{a_1+1}{2}\quad \frac{a_1+2}{2} \dots \frac{a_p+1}{2}\quad \frac{a_p+2}{2}}{\frac{c_1+1}{2}\quad \frac{c_1+2}{2}\dots \frac{c_q+1}{2}\quad \frac{c_q+2}{2}\quad \frac32};4^{p-q-1}x^2}.\nonumber 
\end{align}
Applying on $\zeta(2)$ we get
$$
\!\! \ _3 F_2\zav{\nadsebou{1\quad 1\quad 1}{2\quad 2};1}=\!\! \ _3 F_2\zav{\nadsebou{1\quad \frac12\quad \frac12}{\frac32\quad \frac32};1}+\frac14 \!\! \ _3 F_2\zav{\nadsebou{1\quad 1\quad 1}{2\quad 2};1},
$$
discovering that
$$
\!\! \ _3 F_2\zav{\nadsebou{1\quad \frac12\quad \frac12}{\frac32\quad \frac32};1}=\frac34\!\! \ _3 F_2\zav{\nadsebou{1\quad 1\quad 1}{2\quad 2};1}=\frac{3}{4}\zeta(2)=\frac{\pi^2}{8}.
$$
And
$$
\!\! \ _3 F_2\zav{\nadsebou{1\quad 1\quad 1}{2\quad 2};-1}=\!\! \ _3 F_2\zav{\nadsebou{1\quad \frac12\quad \frac12}{\frac32\quad \frac32};1}-\frac14 \!\! \ _3 F_2\zav{\nadsebou{1\quad 1\quad 1}{2\quad 2};1},
$$
thus
$$
\eta(2)=\!\! \ _3 F_2\zav{\nadsebou{1\quad 1\quad 1}{2\quad 2};-1}=\frac{\pi^2}{8}-\frac{\pi^2}{24}=\frac{\pi^2}{12}.
$$
\end{example}
\begin{remark} The same trick as in the previous example can be used to establish that Catalan's constant $G$ (\ref{Catalan}) is
$$
G=\Re \zav{\!\! \ _3 F_2\zav{\nadsebou{1\quad 1\quad 1}{2\quad 2};\imag}}=\Im\zav{\hzav{\epsilon^2}\!\! \ _2 F_1\zav{\nadsebou{\epsilon\quad \epsilon}{1};\imag}}.
$$
\end{remark}

Allowing ourselves to differentiate with respect to a parameter, we can now prove  Theorem \ref{HFFTCln}.
\begin{proof} By per partes:
\begin{align*}
\int\frac{1}{x}f\zav{ x^\alpha}{\rm d}x&=\ln x f\zav{ x^\alpha}-\alpha [\epsilon]\int x^{\alpha-1+\epsilon} f'\zav{x^\alpha}{\rm d}x\stackrel{(\ref{HFFTCformula})}{=}
\ln x f\zav{x^\alpha}-\hzav{\epsilon}\alpha \frac{x^{\alpha+\epsilon}}{\alpha+\epsilon} f'\zav{\hypzav{\nadsebou{1+\frac{\epsilon}{\alpha}}{2+\frac{\epsilon}{\alpha}}}x^\alpha}\\
&=\ln x f\zav{x^\alpha}- x^\alpha\zav{\ln x-\frac{1}{\alpha}} f'\zav{\hypzav{\nadsebou{1}{2}}x^\alpha}-\frac{x^\alpha}{\alpha} \hzav{\epsilon}f'\zav{\hypzav{\nadsebou{1+\epsilon}{2+\epsilon}}x^\alpha}.
\end{align*}
The theorem now follows from the fact that
$$
f'\zav{\hypzav{\nadsebou{1}{2}}x^\alpha}\stackrel{(\ref{Taylrem})}{=}\frac{f\zav{x^\alpha}-f(0)}{x^\alpha}.
$$
\end{proof}

\begin{example} As an application of this let us verify the well known fact about Catalan's constant:
$$
G=\!\! \ _3 F_2\zav{\nadsebou{1\quad \frac12\quad \frac12}{\frac32 \quad \frac32};-1}=\frac{1}{8}\zav{\psi'\zav{\frac14}-\pi^2}.
$$
\textit{Proof:}
$$
\!\! \ _3 F_2\zav{\nadsebou{1\quad \frac12\quad \frac12}{\frac32 \quad \frac32};-1}=\inte{0}{1}\!\! \ _2 F_1\zav{\nadsebou{1\quad \frac12}{ \frac32};-x^2}{\rm d}x=\inte{0}{1}\frac{{\rm arctan}(x)}{x}{\rm d}x
$$
We are going to use Theorem \ref{HFFTCln} with data $\alpha=1$ and
$$
f(x):=\arctan(x),\qquad f'(x)=\frac{1}{1+x^2},\qquad f'\zav{\hypzav{\nadsebou{a}{b}}x}\stackrel{(\ref{Pochsqr})}{=}\!\! \ _3 F_2\zav{\nadsebou{1\quad \frac{a}{2}\quad \frac{a+1}{2}}{\frac{b}{2}\quad \frac{b+1}{2}};-x^2}.
$$
Hence
$$
\int\frac{\arctan{(x)}}{x}{\rm d}x=\arctan(x)-x\hzav{\epsilon}\!\! \ _2 F_1\zav{\nadsebou{1\quad \frac{1+\epsilon}{2}}{\frac{3+\epsilon}{2}};-x^2}.
$$
Thus
$$
\inte{0}{1}\frac{\arctan{(x)}}{x}{\rm d}x=\frac{\pi}{4}-\hzav{\epsilon}\!\! \ _2 F_1\zav{\nadsebou{1\quad \frac{1+\epsilon}{2}}{\frac{3+\epsilon}{2}};-1}.
$$
Using the summation formula \cite[15.4.27]{dlmf15.4}:
$$
\!\! \ _2 F_1\zav{\nadsebou{1\quad a}{a+1};-1}=\frac{a}{2}\zav{\psi\zav{\frac{a+1}{2}}-\psi\zav{\frac{a}{2}}},
$$
we get
$$
[\epsilon]
\!\! \ _2 F_1\zav{\nadsebou{1\quad \frac{1+\epsilon}{2}}{\frac{3+\epsilon}{2}};-1}=
[\epsilon] \frac{1+\epsilon}{4}\zav{\psi\zav{\frac{3+\epsilon}{4}}-\psi\zav{\frac{1+\epsilon}{4}}}
$$
$$
=\frac14\zav{\psi\zav{\frac34}-\psi\zav{\frac14}}+\frac{1}{16}\zav{\psi'\zav{\frac34}-\psi'\zav{\frac14}}.
$$
The rest follows from reflection formulas for $\psi,\psi'$ function:
$$
\psi(1-x)-\psi(x)=\pi{\rm cot}(\pi x),\qquad \psi'(1-x)+\psi'(x)=\zav{\frac{\pi}{\sin\pi x}}^2.
$$
\end{example}
\subsection{General rules}
We are now going to prove Proposition \ref{genrule1prop}:
\begin{proof}
Using the Pfaff transform (\ref{Pfaff}) twice and the Leibniz rule we obtain
$$
\hzav{\epsilon}\!\! \ _2 F_1\zav{\nadsebou{a+\epsilon\quad b+\epsilon}{a+b};x}\stackrel{(\ref{Pfaff})^2}{=}
\hzav{\epsilon}(1-x)^{-2\epsilon}\!\! \ _2 F_1\zav{\nadsebou{a-\epsilon\quad b-\epsilon}{a+b};x}
$$
$$
=2\ln\frac{1}{1-x}\!\! \ _2 F_1\zav{\nadsebou{a\quad b}{a+b};x}-
\hzav{\epsilon}\!\! \ _2 F_1\zav{\nadsebou{a+\epsilon\quad b+\epsilon}{a+b};x}.
$$
Collecting like terms gives us the result:
$$
\hzav{\epsilon}\!\! \ _2 F_1\zav{\nadsebou{a+\epsilon\quad b+\epsilon}{a+b};x}=\ln\frac{1}{1-x}\!\! \ _2 F_1\zav{\nadsebou{a\quad b}{a+b};x}.
$$
\end{proof}
\begin{example}
Proposition \ref{genrule1prop} act only on hypergeometric function of the form
$$
\!\! \ _2 F_1\zav{\nadsebou{a\quad b}{a+b};x},
$$
i.e. on those function whose sum of upper parameters is equal to the lower parameter. Notable members of this class are
$$
\arctan\, x=x\!\! \ _2 F_1\zav{\nadsebou{1\quad \frac12}{\frac32};-x^2},
$$
and 
$$
K(x)=\frac{\pi}{2}\!\!\ _2 F_1\zav{\nadsebou{\frac12\quad \frac12}{1};x^2},
$$
i.e. the complete elliptic integral of the first kind, which is perhaps the most ``hypergeometric'' function there is, because its special parameters makes it eligible for more quadratic transforms then any other hypergeometric function (to the author's best knowledge).

Proposition \ref{genrule1prop} applied on these functions gives us the following representations:
$$
\ln\frac{1}{1+x^2}\arctan\,x=[\epsilon] x\!\! \ _2 F_1\zav{\nadsebou{1+\epsilon\quad \frac12+\epsilon}{\frac32};-x^2},
$$
and
$$
\ln\frac{1}{1-x}K(x)=\hzav{\epsilon}\frac{\pi}{2}\!\!\ _2 F_1\zav{\nadsebou{\frac12+\epsilon\quad \frac12+\epsilon}{1};x^2},
$$
which allows us easily to compute multiple integrals. First, for $0<\alpha<\frac12$: 
\begin{align*}
\inte{0}{\infty}\frac{\arctan\, x}{x^{2\alpha+1}}\ln\frac{1}{1+x^2}{\rm d}x&=
\inte{0}{\infty} x^{-2\alpha}\hzav{\epsilon}\!\! \ _2 F_1\zav{\nadsebou{1+\epsilon\quad \frac12+\epsilon}{\frac32};-x^2}{\rm d}x\\
&\stackrel{(\ref{HFFTCformula})}{=}
\hzav{\frac{x^{-2\alpha+1}}{1-2\alpha}\hzav{\epsilon}\!\! \ _3 F_2\zav{\nadsebou{1+\epsilon\quad \frac12+\epsilon\quad \frac12-\alpha}{\frac32\quad \frac32-\alpha};-x^2}}_{0}^{\infty}\\
&\stackrel{(\ref{valueinfty})}{=} \frac{1}{1-2\alpha}\frac{\Gamma\zav{\frac12+\epsilon+\alpha}\Gamma\zav{\epsilon+\alpha}\Gamma\zav{\frac32}\Gamma\zav{\frac32-\alpha}}{\Gamma\zav{1+\epsilon}\Gamma\zav{\frac12+\epsilon}\Gamma\zav{1+\alpha}\Gamma\zav{1}}\\
&=\frac{\pi}{4\alpha\cos(\pi\alpha)}\zav{\psi\zav{\frac12+\alpha}+\psi(\alpha)-\psi(1)-\psi\zav{\frac12}}.
\end{align*}

Second:
\begin{align*}
\inte{0}{1}x\ln\frac{1}{1-x^2}K(x){\rm d}x&=\inte{0}{1}x\hzav{\epsilon}\frac{\pi}{2}\!\!\ _2 F_1\zav{\nadsebou{\frac12+\epsilon\quad \frac12+\epsilon}{1};x^2}{\rm d}x\\
&\stackrel{(\ref{HFFTCformula})}{=}\hzav{\frac{\pi}{2}\hzav{\epsilon} \frac{x^2}{2}\!\!\ _3 F_2\zav{\nadsebou{\frac12+\epsilon\quad \frac12+\epsilon\quad 1}{1\quad 2};x^2}}^{1}_{0}\\
&=\frac{\pi}{4}\hzav{\epsilon}\!\!\ _2 F_1\zav{\nadsebou{\frac12+\epsilon\quad \frac12+\epsilon}{2};1}\stackrel{(\ref{2F11})}{=}\frac{\pi}{4}\hzav{\epsilon}\frac{\Gamma(2)\Gamma(1-2\epsilon)}{\Gamma\zav{\frac32-\epsilon}\Gamma\zav{\frac32-\epsilon}}\\
&=2\psi\zav{\frac32}-2\psi(1)=4(1-\ln 2).
\end{align*}

And third:
\begin{align*}
\inte{0}{1}x\ln\frac{1}{1+x^2}K(\imag x){\rm d}x&
=\inte{0}{1}x\frac{\pi}{2}\hzav{\epsilon}\!\! \ _2 F_1\zav{\nadsebou{\frac12+\epsilon\quad \frac12+\epsilon}{1};-x^2}{\rm d}x\\
&\stackrel{(\ref{HFFTCformula})}{=}\hzav{\frac{\pi}{2}\hzav{\epsilon} \frac{x^2}{2}\!\!\ _3 F_2\zav{\nadsebou{\frac12+\epsilon\quad \frac12+\epsilon\quad 1}{1\quad 2};-x^2}}^{1}_{0}\\
&=\frac{\pi}{4}\hzav{\epsilon}\!\!\ _2 F_1\zav{\nadsebou{\frac12+\epsilon\quad \frac12+\epsilon}{2};-1}.\\
\end{align*}
There is no summation formula readily available for this case. But differentiating the transform \cite[15.8.24]{dlmf15.8} with respect to $z$ at $z=-1$ gives us
$$
\frac{ab}{1+a-b}\!\! \ _2F_1\zav{\nadsebou{a+1\quad b+1}{2+a-b};-1}=\frac{a}{2^{a+1}}\frac{\Gamma(a-b+1)\Gamma\zav{\frac12}}{\Gamma\zav{\frac{a+1}{2}}\Gamma\zav{\frac{a}{2}-b+1}}+\frac{1}{2^{a+1}}\frac{\Gamma(a-b+1)\Gamma\zav{-\frac12}}{\Gamma\zav{\frac{a}{2}}\Gamma\zav{\frac{a}{2}-b+\frac12}}.
$$ 
Dividing by the factor $\frac{ab}{1+a-b}$ and
substituting $b=a=-\frac12+\epsilon$ we obtain
$$
\!\! \ _2F_1\zav{\nadsebou{\frac12+\epsilon\quad \frac12+\epsilon}{2};-1}=\frac{\sqrt{\pi}}{\zav{\epsilon-\frac12}^2 2^{\frac12+\epsilon}}\zav{\frac{\epsilon-\frac12}{\Gamma\zav{\frac14+\frac{\epsilon}{2}}\Gamma\zav{\frac54-\frac{\epsilon}{2}}}-\frac{2}{\Gamma\zav{-\frac14+\frac{\epsilon}{2}}\Gamma\zav{\frac34-\frac{\epsilon}{2}}}}.
$$
Differentiating with respect to $\epsilon$ at $\epsilon=0$, multiplying by $\frac{\pi}{4}$ and invoking reflection formula for $\Gamma$ function several times gives us the final result:

\begin{align*}
\inte{0}{1}x\ln\frac{1}{1+x^2}K(\imag x){\rm d}x&=\frac{1}{4\sqrt{2\pi}}\zav{(2-\ln 2)\Gamma^2\zav{\frac14}+4(\ln 2-4)\Gamma^2\zav{\frac34}}.
\end{align*} 
\end{example}

\section{The integral $I_\alpha$}\label{S5}
We now attempt to compute the integral
$$
I_\alpha:=\inte{0}{\infty}\frac{1}{\zav{1+x^2}^{\frac32}}\zav{\varphi(x)+\sqrt{\varphi(x)}}^{-\alpha}{\rm d}x,\qquad \varphi(x):=1+\frac43\frac{x^2}{(1+x^2)^2}.
$$
Making the change of variable:
$$
\frac{x}{\sqrt{1+x^2}}=t,\qquad \varphi(x)=1+\frac43 t^2(1-t^2), \qquad (1+x^2)^{-\frac32}{\rm d}x={\rm d}t,
$$
we obtain
\begin{equation}\label{integralform}
I_\alpha=\inte{0}{1}\zav{\sqrt{\varphi}+\varphi}^{-\alpha}{\rm d}t=\inte{0}{1}\varphi^{-\frac{\alpha}{2}}\zav{1+\sqrt{\varphi}}^{-\alpha}{\rm d}t.
\end{equation}
Using (\ref{gensqrrep}) we arrive at the form
\begin{align}\label{integralform1}
I_\alpha&=2^{-\alpha}\inte{0}{1}\zav{1+\frac43 t^2(1-t^2)}^{-\frac{\alpha}{2}}\!\! \ _2 F_1\zav{\nadsebou{\frac{\alpha}{2}\quad \frac{\alpha+1}{2}}{\alpha+1};-\frac43 t^2(1-t^2)}{\rm d}t,\\
&=2^{-\alpha}\inte{0}{1}\zav{1+\frac43 t^2(1-t^2)}^{-\frac{\alpha-1}{2}}\!\! \ _2 F_1\zav{\nadsebou{1+\frac{\alpha}{2}\quad \frac{\alpha+1}{2}}{\alpha+1};-\frac43 t^2(1-t^2)}{\rm d}t,\label{integralform2}
\end{align}
where the equality stems from applying Pfaff transform (\ref{Pfaff}) twice.

The integrand is a product of two hypergeometric functions, which cannot (in general) be represented as a single hypergeometric function. Proposition \ref{HFFTC} tells us that the antiderivative is a single variable hypergeometric function \textit{if and only if} the integrand is.

Hence we can describe the result only in terms of a multi-variable function. Specifically, the $\tilde F_1$ function:  

Applying Proposition \ref{intrep2} together with the definition of $\tilde F_1$ (\ref{tildeF1}) we get:
\begin{align}
I_\alpha&=2^{-\alpha} \tilde F_1\zav{\nadsebou{1\quad \frac12}{\frac34\quad \frac54};\nadsebou{\frac{\alpha}{2}\quad \frac{\alpha+1}{2}}{\alpha+1}\nadsebou{\frac{\alpha}{2}}{-};-\frac13,-\frac13}\nonumber \\
&=2^{-\alpha} \tilde F_1\zav{\nadsebou{1\quad \frac12}{\frac34\quad \frac54};\nadsebou{1+\frac{\alpha}{2}\quad \frac{\alpha+1}{2}}{\alpha+1}\nadsebou{\frac{\alpha-1}{2}}{-};-\frac13,-\frac13}. \label{tildeF1rep}
\end{align}

More illuminating approach is perhaps to use hypergeometrization:
For 
$$
f_\alpha(x):=(1-x)^{-\frac{\alpha}{2}}\!\! \ _2 F_1\zav{\nadsebou{\frac{\alpha}{2}\quad \frac{\alpha+1}{2}}{\alpha+1};x}=(1-x)^{-\frac{\alpha-1}{2}}\!\! \ _2 F_1\zav{\nadsebou{1+\frac{\alpha}{2}\quad \frac{\alpha+1}{2}}{\alpha+1};x},
$$
we have
$$
I_\alpha=2^{-\alpha}f_\alpha\zav{\hypzav{\nadsebou{1\quad \frac12}{\frac34\quad \frac54}}-\frac13}.
$$

We now prove Proposition \ref{Ialpha}.
\begin{proof}
\textit{Case }$\alpha=0$. Obviously $I_0=1$.

\bigskip
\textit{Case }$\alpha=1$:
We can immediately see that 
$$
f_{1}(x)=\!\! \ _2 F_1\zav{\nadsebou{\frac32\quad 1}{2};x},
$$
thus
$$
I_1=\frac12\!\! \ _4 F_3\zav{\nadsebou{1\quad \frac12\quad\frac32\quad 1}{2\quad \frac34\quad \frac54};-\frac13}=-\frac38 \hzav{\epsilon} \!\! \ _3 F_2\zav{\nadsebou{\epsilon \quad -\frac12\quad\frac12}{ -\frac14\quad \frac14};-\frac13}.
$$

\bigskip
\textit{Case }$\alpha=-1$:
Similarly
$$
f_\alpha(x)=(1-x)^{-\frac{\alpha}{2}}\!\! \ _2 F_1\zav{\nadsebou{\frac{\alpha}{2}\quad \frac{\alpha+1}{2}}{\alpha+1};x}\stackrel{(\ref{Taylrem})}{=}(1-x)^{-\frac{\alpha}{2}}\zav{1+\frac{\alpha}{4}x\!\! \ _3 F_2\zav{\nadsebou{\frac{\alpha+2}{2}\quad \frac{\alpha+3}{2}\quad 1}{\alpha+2\quad 2};x}},
$$
hence
$$
f_{-1}(x)=(1-x)^{\frac12}\zav{1-\frac{x}{4}\!\! \ _2 F_1\zav{\nadsebou{1\quad \frac12}{2};x}}\stackrel{(\ref{Taylrem})}{=}\frac{(1-x)^{\frac12}+1-x}{2},
$$
and thus we have
$$
I_{-1}=\!\!\ _3 F_2\zav{\nadsebou{1\quad \frac12\quad -\frac12}{\frac34\quad \frac54};-\frac13}+\frac{53}{45}.
$$

\bigskip
\textit{Case }$\alpha=-n$:
In fact, all $I_{-n}$ for $n\in\mathbb{N}$ can reduced to a linear combination of hypergeometric functions:
$$
I_{-n}=\sum_{k=0}^{n}\zav{\nadsebou{n}{k}}\!\! \ _3 F_2\zav{\nadsebou{-\frac{n+k}{2}\quad 1\quad \frac12}{\frac34\quad \frac54};-\frac13}.
$$
This follows immediately by expanding the term $(1+\sqrt{\varphi})^{n}$ in the integral (\ref{integralform}).

\bigskip
\textit{Case }$\alpha=2$:
$$
f_2(x)=(1-x)^{-1}\!\! \ _2 F_1\zav{\nadsebou{1\quad \frac32}{3};x}.
$$
Presence of the parameters $\nadsebou{1}{3}$ in the $\!\! \ _2 F_1$ function suggests that it can be understood as a Taylor reminder of a simpler function. And, indeed, that is the case:
$$
-\frac{x^2}{8}\!\! \ _2 F_1\zav{\nadsebou{1\quad \frac32}{3};x}\stackrel{(\ref{Taylrem})}{=}(1-x)^{\frac12}-1+\frac{x}{2}.
$$
Thus
$$
-\frac{x^2}{8}(1-x)^{-1}\!\! \ _2 F_1\zav{\nadsebou{1\quad \frac32}{3};x}=(1-x)^{-\frac12}-\zav{1-\frac{x}{2}}(1-x)^{-1}\stackrel{(\ref{Taylrem})}{=}\frac{x^2}{2}+\frac{x^2}{2}(x-2)(1-x)^{-1}+\frac{3x^2}{8}\!\! \ _2F_1\zav{\nadsebou{\frac52\quad 1}{3};x},
$$
hence we arrive at the following representation:
$$
f_{2}(x)=(1-x)^{-1}\!\! \ _2 F_1\zav{\nadsebou{1\quad \frac32}{3};x}=-4-4(x-2)\!\! \ _1 F_0\zav{\nadsebou{1}{-};x}-3\!\! \ _2F_1\zav{\nadsebou{\frac52\quad 1}{3};x}.
$$
Making hypergeometrization and dividing by 4 we get
$$
I_2=-1+2\!\! \ _3 F_2\zav{\nadsebou{1\quad 1\quad \frac12}{\frac34\quad \frac54};-\frac13}+\frac{8}{45}\!\! \ _3 F_2\zav{\nadsebou{1\quad 2\quad \frac32}{\frac74\quad \frac94};-\frac13}-\frac34 \!\! \ _4
 F_3\zav{\nadsebou{1\quad 1\quad \frac12\quad \frac52}{\frac34\quad \frac54\quad 3};-\frac13}.
$$
With the aid of formula:
\begin{equation}\label{aidfor}
\!\! \ _3 F_2\zav{\nadsebou{1\quad \frac{a}{2}\quad \frac{a+1}{2}}{\frac{c}{2}\quad \frac{c+1}{2}};-x^2}=\Re\zav{\!\! \ _2 F_1\zav{\nadsebou{1\quad a}{c};\imag x}},
\end{equation}
which stems from (\ref{evenodd}) we can simplify first two terms:
$$
I_2=-1+\Re\zav{2\!\! \ _2 F_1\zav{\nadsebou{1\quad 1}{\frac32};\frac{\imag}{\sqrt{3}}}+\frac{8}{45}\!\! \ _2 F_1\zav{\nadsebou{1\quad 3}{\frac72};\frac{\imag}{\sqrt{3}}}}-\frac34 \!\! \ _4
 F_3\zav{\nadsebou{1\quad 1\quad \frac12\quad \frac52}{\frac34\quad \frac54\quad 3};-\frac13}.
$$
Since for $z=x+\imag y$ it holds:
$$
\Re\zav{\!\! \ _2 F_1\zav{\nadsebou{1\quad 1}{\frac32};\frac{z^2}{z^2+1}}}\stackrel{(\ref{Pfaff})}{=}\Re\zav{\frac{1+z^2}{z}\arctan\ z}=\frac{\abs{z}^2-1}{\abs{z}^2}\zav{\frac{y}{4}\ln\frac{1-2y+\abs{z}^2}{1+2y+\abs{z}^2}+\frac{x}{2}\arctan\frac{2x}{1-\abs{z}^2}},
$$
we get
$$
\Re\zav{\!\! \ _2 F_1\zav{\nadsebou{1\quad 1}{\frac32};\frac{\imag}{\sqrt{3}}}}=\frac{1}{8\sqrt{2}}\zav{6\,\arctan\sqrt{2}-\sqrt{3}\ln\zav{5-2\sqrt{6}}}.
$$

Similarly
$$
\!\! \ _2 F_1\zav{\nadsebou{1\quad 3}{\frac72};\frac{z^2}{z^2+1}}\stackrel{(\ref{Pfaff})}{=}(1+z^2)^{3}\!\! \ _2 F_1\zav{\nadsebou{\frac52\quad 3}{\frac72};-z^2}\stackrel{(\ref{HFFTCformula})}{=}\frac{(1+z^2)^3 5}{z^5}\int \frac{z^4}{(1+z^2)^{3}}{\rm d}z
$$
$$
=\frac{15}{8} \frac{(1+z^2)^3}{z^5}\arctan\,z-\frac{5}{32}(3+5z^2)(1+z^2).
$$
Thus
$$
\Re\zav{\!\! \ _2 F_1\zav{\nadsebou{1\quad 3}{\frac72};\frac{\imag}{\sqrt{3}}}}=\frac{45}{8}-\frac{135\sqrt{2}}{64}\arctan\sqrt{2}-\frac{45}{128}\sqrt{6}\ln\zav{5+2\sqrt{6}}.
$$
Combining these two formulas gives as the desired result.

\bigskip
\textit{Case }$\hzav{\alpha}$:
$$
f_\alpha(x)=(1-x)^{-\frac{\alpha}{2}}\!\! \ _2 F_1\zav{\nadsebou{\frac{\alpha}{2}\quad \frac{\alpha+1}{2}}{\alpha+1};x}\stackrel{(\ref{Taylrem})}{=}
\zav{1+\alpha\frac12 \ln\frac{1}{1-x}+O(\alpha^2)}\zav{1+\frac{\alpha}{4}x\!\! \ _3 F_2\zav{\nadsebou{\frac{\alpha}{2}+1\quad \frac{\alpha+3}{2}\quad 1}{\alpha+2\quad 2};x}},
$$
thus
$$
\hzav{\alpha} f_\alpha(x)=\frac{1}{2}\ln\frac{1}{1-x}+\frac{x}{4}\!\! \ _3 F_2\zav{\nadsebou{1\quad \frac{3}{2}\quad 1}{2\quad 2};x}.
$$
Since
$$
\hzav{\alpha}I_\alpha=-\ln 2 I_0+\hzav{\alpha} f_\alpha\zav{\hypzav{\nadsebou{1 \quad \frac12}{\frac34\quad \frac54}}-\frac13},
$$
we arrive
$$
\hzav{\alpha}I_\alpha=\ln\frac12-\frac{4}{45}\!\! \ _3 F_2\zav{\nadsebou{1\quad 1\quad \frac32}{\frac74\quad \frac94};-\frac13}-\frac{2}{45}\!\! \ _4 F_3\zav{\nadsebou{1\quad 1\quad \frac32\quad \frac32}{2\quad\frac74\quad \frac94};-\frac13}.
$$

Similarly as in the previous case, the first term can be simplified, noting that
$$
\!\! \ _3 F_2\zav{\nadsebou{1\quad 1\quad \frac32}{\frac74\quad \frac94};-\frac13}=\Re\zav{\!\! \ _2 F_1\zav{\nadsebou{1\quad 2}{\frac72};\frac{\imag}{\sqrt{3}}}}.
$$
Since
$$
\!\! \ _2 F_1\zav{\nadsebou{1\quad 2}{\frac72};\frac{z^2}{z^2+1}}\stackrel{(\ref{Pfaff})}{=}(1+z^2)^{2}\!\! \ _2 F_1\zav{\nadsebou{\frac52\quad 2}{\frac72};-z^2}\stackrel{(\ref{HFFTCformula})}{=}\frac{(1+z^2)^2 5}{z^5}\int \frac{z^4}{(1+z^2)^{2}}{\rm d}z
$$
$$
= \frac{5(1+z^2)^2}{z^5}\zav{z+\frac12\frac{z}{1+z^2}-\frac32\arctan\,z}.
$$
Thus
$$
\Re\zav{\!\! \ _2 F_1\zav{\nadsebou{1\quad 2}{\frac72};\frac{\imag}{\sqrt{3}}}}=-\frac{45}{2}+\frac{45\sqrt{2}}{8}\arctan\sqrt{2}+\frac{45}{16}\sqrt{6}\ln\zav{5+2\sqrt{6}}.
$$
Substituting this into gives us the desired result.

\end{proof}

Now Proposition \ref{P1} and \ref{P2}.
\begin{proof}
Unfortunately, we are unable to similarly reduce the function 
$$
f_{\frac12}(x)=(1-x)^{-\frac14} \!\! \ _2 F_1\zav{\nadsebou{\frac{1}{4}\quad \frac{3}{4}}{\frac32};x},
$$ 
for the original integral $I_{\frac12}$. An obvious idea is to eliminate troublesome term $(1-x)^{-\frac14}$ by Euler transform, but the power is not right.

If instead of $\frac14$ there was $\frac12$ then we would have single hypergeometric function 
$$
f_{true}(x):=(1-x)^{-\frac14} f_{\frac12}(x)=\!\! \ _2 F_1\zav{\nadsebou{\frac{5}{4}\quad \frac{3}{4}}{\frac32};x}.
$$
Multiplying $f_\frac12$ by $(1-x)^{-\frac14}$ is the same as multiplying the integrand in $I_{\frac12}$ by $\varphi^{-\frac14}$ hence we are dealing with the integral $I_{true}$.

For this case we have
$$
I_{true}=\frac{1}{\sqrt{2}} f_{true}\zav{\hypzav{\nadsebou{1\quad \frac12}{\frac34\quad \frac54}}-\frac13}=\frac{1}{\sqrt{2}} \!\! \ _4 F_3\zav{\nadsebou{1\quad \frac12\quad\frac{5}{4}\quad \frac{3}{4}}{\frac34\quad \frac54\quad\frac32};-\frac13}=\frac{1}{\sqrt{2}}
\!\! \ _2 F_1\zav{\nadsebou{1\quad \frac12}{\frac32};-\frac13}
$$
$$
=\frac{1}{\sqrt{2}}\frac{\arctan{\frac{1}{\sqrt{3}}}}{\frac{1}{\sqrt{3}}}=\frac{1}{\sqrt{2}}\frac{\frac{\pi}{6}}{\frac{1}{\sqrt{3}}}=\frac{\pi}{2\sqrt{6}},
$$
thus proving Proposition \ref{P1}.

Proof of Proposition \ref{P2} requires only to realize that 
$$
I_{\alpha,true}= \frac{1}{\sqrt{2}} f_{true}\zav{\hypzav{\nadsebou{1\quad \frac12}{\frac34\quad \frac54}}-\alpha^2}.
$$
\end{proof}


\begin{thebibliography}{1}
\bibitem{mobref1} T. Amdeberhan, V. Moll:
\textit{A formula for a quartic integral: a survey of old proofs and some new ones.} Ramanujan J., 18 (2009), pp. 91-102 
\bibitem{dlmfref3} W. N. Bailey (1964): \textit{Generalized Hypergeometric Series.} Stechert-Hafner, Inc., New York.
\bibitem{Erdelyi} H. Bateman, A. Erdelyi, Higher transcendental functions, vol. 1, McGraw-Hill Book Co., New York, 1953.
\bibitem{Baker} Baker, Alan (1990), Transcendental number theory, Cambridge Mathematical Library (2nd ed.), Cambridge University Press, ISBN 978-0-521-39791-9, MR 0422171
\bibitem{B2} P. Blaschke: {\it Berezin transform on harmonic Bergman spaces on the real ball}, J. Math. Anal. Appl. 411 (2014), no. 2, 607-630.
\bibitem{mobref2} E.E. Boos, A.I. Davydychev:
\textit{A method of evaluating massive Feynman integrals},
Theoret. and Math. Phys., 89 (1991), pp. 1052-1063
\bibitem{dlmfref1} W. Bühring (1992): \textit{Generalized hypergeometric functions at unit argument.} Proc. Amer. Math. Soc. 114 (1), pp. 145–153.
\bibitem{Carlson} Carlson, B. C., Shaffer, D. B.: Starlike and prestarlike hypergeometric functions. SIAM J. Math. Anal. 159, 737–745 (1984) MR0747433 (85j:30014)
\bibitem{Cherry1} G. W. Cherry. \textit{Integration in finite terms with special functions: the error function.} J. Symb. Comput., 1:283-302,1985.
\bibitem{Cherry2} G. W. Cherry. \textit{Integration in finite terms with special functions: the logarithmic function.} SIAM J. Comput., 15:1-21,1986.
\bibitem{Cherry3} G. W. Cherry. \textit{An analysis of the rational exponential integral.} J. Symb. Comput., SIAM J. Comput., 18:893-905,1989.
\bibitem{dlmfref4} A. Cuyt, V. B. Petersen, B. Verdonk, H. Waadeland, W. B. Jones (2008): \textit{Handbook of Continued Fractions for Special Functions}. Springer, New York.
\bibitem{mob} I. Gonzales, V.H. Moll: \textit{Definite integrals by the method of brackets. Part 1}, Advances in Applied Mathematics, Volume 45, Issue 1, July 2010, Pages 50-73, \url{https://doi.org/10.1016/j.aam.2009.11.003}
\bibitem{mobref3} I. Gonzalez, I. Schmidt:\textit{
Recursive method to obtain the parametric representation of a generic Feynman diagram.} Phys. Rev. D, 72 (2005), p. 106006
\bibitem{GR} I.S. Gradshteyn, I.M. Ryzhik. \textit{Table of Integrals, Series, and Products.} Edited by A. Jeffrey and D. Zwillinger. Academic Press, New York, 7th edition, 2007.
\bibitem{dlmfref2} Y. S. Kim, A. K. Rathie, R. B. Paris (2013): \textit{An extension of Saalschütz’s summation theorem for the series} $\!\! \ _{r+3}F_{r+2}$. Integral Transforms Spec. Funct. 24 (11), pp. 916–921
\bibitem{Lauricella} Lauricella, Giuseppe (1893). {\it Sulle funzioni ipergeometriche a pi\`u variabili}. Rendiconti del Circolo Matematico di Palermo (in Italian) 7 (S1): 111–158. doi:10.1007/BF03012437. JFM 25.0756.01.
\bibitem{Luke} Y.L. Luke, The special functions and their approximations, Academic Press, 1969. MR0241700 (39 \# 3039).
\bibitem{Olver}F. W. J. Olver: {\it Asymptotics and Special Functions}, Academic Press, New York, 1974.
\bibitem{Srivastava} H.M. Srivastava, P.W. Karlsson: Multiple Gaussian Hypergeometric Series, Halsted Press (Ellis Horwood Limited, Chichester), John Wiley and Sons, New York, London, Toronto, 1985. MR0834385
\bibitem{dlmf} \url{http://dlmf.nist.gov}
\bibitem{dlmfeq2} \url{http://dlmf.nist.gov/16.2}
\bibitem{dlmfeq} \url{http://dlmf.nist.gov/16.11}
\bibitem{dlmf16.4} \url{http://dlmf.nist.gov/16.4}
\bibitem{dlmf15.4} \url{http://dlmf.nist.gov/15.4}
\bibitem{dlmf15.8} \url{http://dlmf.nist.gov/15.8}
\end{thebibliography}
\end{document}